\documentclass[12pt]{amsart}
\usepackage{amsmath,amsfonts,amssymb,amsthm,tikz,tikz-cd,color,todonotes,epigraph,mathrsfs,graphicx}

\newtheorem{Theorem}{Theorem}[section]

\newtheorem{Lemma}[Theorem]{Lemma}
\newtheorem{Question}[Theorem]{Question}

\newtheorem{Conjecture}[Theorem]{Conjecture}

\theoremstyle{definition}
\newtheorem{Definition}[Theorem]{Definition}
\newtheorem{Example}[Theorem]{Example}
\newtheorem{Remark}[Theorem]{Remark}

\newcommand{\C}{{\mathbb C}}

\newcommand{\N}{{\mathbb N}}
\newcommand{\Z}{{\mathbb Z}}
\newcommand{\R}{{\mathbb R}}
\newcommand{\Cx}{{\C^\times}}
\newcommand{\ft}{{\mathfrak t}}
\newcommand{\fg}{{\mathfrak g}}
\newcommand{\fh}{{\mathfrak h}}

\newcommand{\cN}{{\mathcal N}}
\newcommand{\cF}{{\mathcal F}}
\newcommand{\gl}{{\mathfrak sl}}
\newcommand{\PP}{{\mathbb P}}
\newcommand{\sslash}{/\!/}
\newcommand{\ssslash}{/\!/\!/}
\newcommand{\Bv}{{\mathbf v}}
\newcommand{\Bw}{{\mathbf w}}
\newcommand{\IC}{{\mathrm{IC}}}
\newcommand{\tY}{{\mathscr{Y}}}
\newcommand{\tX}{{\mathscr{X}}}
\newcommand{\cA}{{\mathcal A}}
\newcommand{\A}{{\mathscr A}}
\newcommand{\cO}{{\mathcal O}}
\newcommand{\cK}{{\mathcal K}}
\newcommand{\cM}{{\mathcal M}}
\newcommand{\cP}{{\mathcal P}}
\newcommand{\ul}{{\underline \la}}
\newcommand{\la}{{\lambda}}
\newcommand{\cW}{{\mathcal W}}
\newcommand{\oW}{{\overline{\cW}}}
\newcommand{\ssl}{{\mathfrak{sl}}}
\newcommand{\Gr}{{\mathsf{Gr}}}
\newcommand{\tG}{\widetilde{G}}
\newcommand{\Dx}{D^\times}
\newcommand{\BB}{{\mathbb B}}
\newcommand{\cV}{{\mathcal V}}

\newcommand{\Kreg}{K^{\mathrm{reg}}}

\newcommand{\Spec}{\operatorname{Spec}}
\newcommand{\Proj}{\operatorname{Proj}}
\newcommand{\Hom}{\operatorname{Hom}}
\newcommand{\rk}{\operatorname{rk}}
\newcommand{\codim}{\operatorname{codim}}
\newcommand{\Sym}{\operatorname{Sym}}
\newcommand{\Irr}{\operatorname{Irr}}
\newcommand{\Maps}{\operatorname{Maps}}
\newcommand{\End}{\operatorname{End}}

\newcommand{\Oof}[1]{{{#1}{\text{-mod}_\cO}}}
\newcommand{\Mof}[1]{{{#1}{\text{-mod}}}}

\newcommand{\Vect}{\mathsf{Vect}}
\newcommand{\Rep}{\mathsf{Rep}}

\author{Joel Kamnitzer}
\address{J.~Kamnitzer: Department of Mathematics, University of Toronto, Canada}
\email{jkamnitz@math.toronto.edu}
\title{Symplectic resolutions, symplectic duality, and Coulomb branches}

\begin{document}
	
\begin{abstract}
	Symplectic resolutions are an exciting new frontier of research in representation theory.  One of the most fascinating aspects of this study is symplectic duality: the observation that these resolutions come in pairs with matching properties.  The Coulomb branch construction allows us to produce and study many of these dual pairs.  These notes survey much recent work in this area including quantization, categorification, and enumerative geometry.  We particularly focus on ADE quiver varieties and affine Grassmannian slices.
\end{abstract}
	
\maketitle
\textit{Dedicated to the memory of Tom Nevins}

\section{Introduction}

\epigraph{Symplectic resolutions are the Lie algebras of the $21^{st}$ century}{Andrei Okounkov}

\subsection{Symplectic resolutions}

In the $20^{th}$ century, mathematicians studied the representation theory of semisimple Lie algebras using the geometry of the flag variety and its cotangent bundle.  In the $21^{st} $ century, we have generalized this study to the setting where the cotangent bundle of the flag variety is replaced by a \textbf{symplectic resolution}.  This is a smooth symplectic variety $ Y $ which resolves an affine Poisson variety $ X $.  In parallel, the semisimple Lie algebra is replaced by the deformation quantization $ A $ of $ X $.  
In this survey, we consider many important examples of symplectic resolutions, including hypertoric varieties, quiver varieties, and affine Grassmannian slices. 

\subsection{Topology and quantization}
We also assume that $ Y, X $ are equipped with two torus actions, one conical, and the other Hamiltonian.  In section \ref{se:top}, we study the equivariant topology of the map $ \pi : Y \rightarrow X $, following the work of Kaledin \cite{Ka}, McGerty-Nevins \cite{MN}, Nakajima \cite{NakPCMI} and others. In particular, we focus here on the \textbf{hyperbolic BBD decomposition}
\begin{equation*} 
	H(Y^+) = \bigoplus_{j \in J} H(\overline{X_j^+}) \otimes H(F_j)
\end{equation*}
which gives a decomposition of the homology of the attracting set $ Y^+ $ (for a Hamiltonian $ \Cx$-action) in terms of the fibres $ F_j$ of $ \pi $ and the symplectic leaves $X_j$ of $ X $.

There is a fascinating interplay between the topology of $ Y $ and the representation theory of its quantization $ A$.  In section \ref{se:alg}, following the ideas of Braden-Licata-Proudfoot-Webster \cite{BLPW}, we introduce category $ \mathcal O $ for $ A$, which we regard as a categorification of $ H(Y^+) $. 

\subsection{Symplectic duality / 3d mirror symmetry}
There is a beautiful duality between pairs of symplectic resolutions, first observed in math by Braden-Licata-Proudfoot-Webster \cite{BLPW}, and closely related to duality for 3d supersymmetric quantum field theories (see Intrilligator-Seiberg \cite{IS}).  Certain structures on one symplectic resolution $ Y $ match other structures on its dual $ Y^!$, see section \ref{se:SD}.  In particular, there is an isomorphism $ H(Y^+) \cong H({Y^!}^+) $ which reverses the hyperbolic BBD decompositions.  An essential recent breakthrough (described in section \ref{se:CB}) is the construction of Coulomb branches by Braverman-Finkelberg-Nakajima \cite{BFN1} which allows us to closely link many symplectic dual pairs.

\subsection{An application}
In the last twenty years of the $20^{th}$ century, mathematicians developed two geometric constructions of the finite-dimensional irreducible representations of a semisimple Lie algebra $ \mathfrak g $.  The first construction involves the cohomology of quiver varieties constructed using the Dynkin diagram of $ \mathfrak g$.  The second uses the geometric Satake correspondence which constructs representations using the affine Grassmannian of the Langlands dual group $G^\vee$.  
\begin{Question}
	What is the relationship between these two geometric constructions?
\end{Question}
In the first twenty years of the $21^{st} $ century, both of these constructions were placed under the framework of symplectic resolutions and the above question was answered using symplectic duality.  This will be our focus in section \ref{se:SDQV}.

\subsection{Acknowlegements}
I would like to thank G. Bellamy, D. Ben-Zvi, R. Bezrukavnikov, A. Braverman, T. Dimofte, M. Finkelberg, D. Gaiotto, V. Ginzburg, J. Hilburn, A. Licata, M. McBreen, K. McGerty, H. Nakajima, T. Nevins, N. Proudfoot, B. Webster, A. Weekes, O. Yacobi for many insipring conversations and lectures on these topics.  

These notes are based on my lectures at the 2021 IHES summer school on Enumerative Geometry, Physics and Representation Theory.  I thank A. Negut, F. Sala, and O. Schiffmann for organizing the summer school and I thank M. McBreen and Y. Zhou for leading discussion sections which accompanied my lectures.  I also thank M. McBreen, H. Nakajima, N. Proudfoot, A. Referee, A. Weekes, and O. Yacobi for helpful comments on these notes.  Finally, I thank B. Kamnitzer for her drawings.

\section{Symplectic resolutions}
 \label{se:def}
\begin{Definition}
	A \textbf{symplectic resolution} is a morphism $ \pi : Y \rightarrow X $ of complex algebraic varieties, where
	\begin{itemize}
		\item $ Y $ is a smooth and carries a symplectic structure,
		\item $ X $ is affine, normal, and carries a Poisson structure,
		 \item $ \pi $ is projective, birational, and Poisson.
		\end{itemize}
	A symplectic resolution is called \textbf{conical} if we are given actions of $ \Cx $ on $ X$ and $ Y $, compatible with $ \pi$, such that $ \Cx $ contracts $ X $ to a single point, denoted $ 0 $.  We also assume that $ \Cx $ scales the symplectic form with weight $ 2 $.  The central fibre $ F_0 = \pi^{-1}(0)$ is called the \textbf{core} of $ Y $. 
	
	A conical symplectic resolution is called \textbf{Hamiltonian} if we are given Hamiltonian actions of a torus $ T $ on $ X$ and $ Y $, such that $ \pi $ is $T$-equivariant.  We also assume that the $ T $ action commutes under the conical $ \Cx $ action, and that $ Y^T $ is finite.
\end{Definition}
\begin{Remark}
	Sometimes it is useful to consider sympletic resolutions which do not have all the above structures.  
\begin{enumerate}
	\item The ``resolution'' $Y$ may not be smooth.  In this case, most of the desired structures will still exist (such as deformations and quantizations), if we assume that $ X $ has \textbf{symplectic singularities} (which means that $ \omega $ extends as a 2-form on any resolution of $ X $) and $ Y $ is a $\mathbb Q$-\textbf{factorial terminalization}, see \cite{Lo}.
	\item There may not be a Hamiltonian torus action with finitely many fixed points.  In this case, the hyperbolic stalk functor described below must be replaced by hyperbolic restriction.
\end{enumerate}
\end{Remark}

\begin{Remark} \label{re:nodeg1}
	The conical $\Cx$-action on $ X $ gives a $ \Z$-grading to its coordinate ring.  Because the action is contracting, $ \C[X]_k = 0 $ for $ k < 0 $ and $ \C[X]_0 = \C $.  In what follows, we will also assume that $ \C[X]_1 = 0 $, as this implies that $ \C[X]_+ $ is a Poisson ideal, and hence $ \{0 \} $ is a symplectic leaf.
	\end{Remark}

We now consider the main examples of symplectic resolutions.

\subsection{Cotangent bundle of $\PP^1$}
	The simplest example of a symplectic resolution is 
	$$ \pi :Y = T^* \PP^1 \rightarrow \cN_{\gl_2} = X $$
	where $ \cN_{\gl_2} $ is the variety of nilpotent 2$\times$2 matrices.
	\begin{figure}
	\includegraphics[trim=0 130 5 80, clip,width=0.8\textwidth]{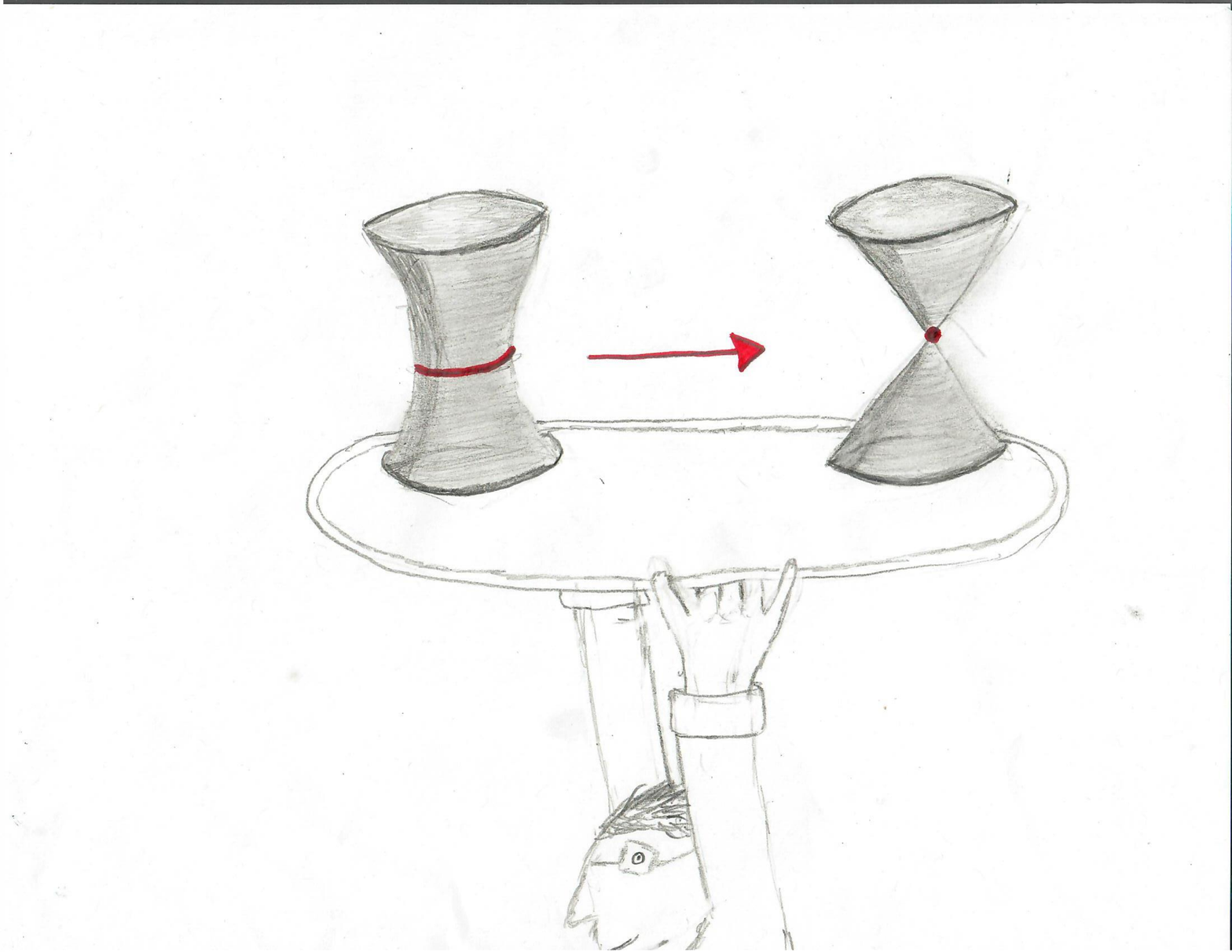}
	\caption{The simplest symplectic resolution, $T^* \PP^1 \rightarrow \cN $.}
	\end{figure}
	We can think of these spaces more explicitly as
	\begin{gather*}
	T^* \PP^1 = \{(A,L) \in \cN_{\gl_2} \times \PP^1 : L \subset \C^2,\, A \C^2 \subset L,\, A L = 0 \} \\
	\cN_{\gl_2} = \Bigl\{ A = \begin{bmatrix} w & u \\ v & -w \end{bmatrix} \in M_2(\C) :  w^2 + uv = 0 \Bigr\} \cong \C^2 / \Z_2
	\end{gather*}

The map $ \pi $ is an isomorphism away from $ 0 $ and has $ \PP^1 $ fibre over $ 0$ (the core).

We have a conical $ \Cx$-action which scales the matrix $ A$ and a Hamiltonian action of $ T= (\Cx)^2 $, inherited from its action on $ \C^2$.  The moment map for this action is given by projection of the matrix $ A $ to its diagonal, so is given by $ (u,v,w) \mapsto (w, -w) $.

\subsection{Cotangent bundles of (partial) flag varieties}
We can generalize to cotangent bundles of other projective spaces. 
$$
Y = T^* \PP^{n-1} \rightarrow X = \{ A \in M_n(\C) : A^2 = 0, \rk A \le 1 \}
$$
Or, more generally, we let $ G $ be any reductive group and $ P $ a parabolic subgroup.  Then the cotangent bundle of the partial flag variety is a resolution of a nilpotent orbit closure (here $ \mathfrak u_P $ is the nilpotent radical of $ \mathfrak p $).
$$
Y = T^* G/P \rightarrow X = \overline{G \mathfrak{u}_P} \subset \cN_\fg
$$
For $ G = SL_n$, every nilpotent orbit closure appears in this way, but this is not true in other types (these orbits are called \textbf{Richardson} orbits).  Namikawa has extensively studied symplectic resolutions of nilpotent orbit closures, see \cite{Nam0809}.

The conical $ \Cx $ acts by linear scaling on $ X$ (coming from its embedding in the vector space $ \fg $).  The maximal torus $ T \subset G $ acts Hamiltonianly on $ Y $ in the natural way with moment map given by the composition $ X \subset \fg = \fg^* \rightarrow \ft^* $.  The fixed point set $ Y^T $ is in bijection with $ W/W_P $.  

We will be particularly interested in the full flag variety for $ SL_n $, so 
\begin{gather*}
Y = T^* Fl_n = \{(A, V_\bullet) : V_1 \subset \dots \subset V_n = \C^n, A V_i \subseteq V_{i-1} \} \rightarrow \\
X = \cN_{\gl_n} = \{A \in M_n(\C) : A \text{ is nilpotent } \} 
\end{gather*}
In this case, the fibres of $ \pi $ are called Springer fibres and $Y \rightarrow X $ is often called the Springer resolution.

In all these examples, the symplectic structure on $ Y $ comes from the fact that it is a cotangent bundle, and the Poisson structure on $ X $ comes from restricting the Kirillov-Kostant-Souriau Poisson structure on $ \fg^* $.

For further variations on this theme, we could consider Slodowy slices and their resolutions, but for lack of space we will restrain ourselves.

\subsection{Resolutions of Kleinian singularities}  Generalizing $ T^* \PP^1 $ in a different direction, we take $ \Gamma \subset SL_2(\C) $ a finite subgroup.  Under the McKay correspondence, such subgroups are in bijection with simply-laced ADE Dynkin diagrams.  For our purposes, the most important case is
$$
\Z_n = \left\{ \begin{bmatrix} \zeta & 0 \\ 0 & \zeta^{-1} \end{bmatrix} : \zeta^n = 1 \right\} \longleftrightarrow A_{n-1}
$$

The affine GIT quotient $ X := \C^2 \sslash \Gamma $ carries a Poisson structure by descending the usual symplectic structure on $ \C^2$.  More explicitly, $ \C[X] = \C[x,y]^\Gamma $ is a Poisson subalgebra of $ \C[x,y] $.

There is a minimal resolution $ Y = \widetilde{\C^2/\Gamma} $ of $ X $ which can be realized by various means (for example by realizing $ X $ as a Slodowy slice).  The map $ \pi : Y \rightarrow X $ is an isomorphism away from 0 and the core $ \pi^{-1}(0) $ is a chain of projective lines which are glued together according to the corresponding Dynkin diagram.

The conical $ \Cx $ action comes from the scaling action on $ \C^2 $.  On the other hand, the Hamiltonian torus $ T$ is given by the diagonal matrices in $SL_2 $ which commute with $ \Gamma$.  When $ \Gamma = \Z_n $, $ T = \Cx $ is the maximal torus $ SL_2 $.  For other $ \Gamma $, $ T $ is trivial.

\subsection{Higgs branches of gauge theories} \label{se:Higgs}
One way to build symplectic resolutions is by using Hamiltonian reduction.  We begin with a reductive group $ G $ and a representation $ N $.  Then we form $ T^*N = N \oplus N^* $ which comes with a moment map $ \Phi : N \oplus N^* \rightarrow \fg^* $.  Then we take $ \Phi^{-1}(0) $ and quotient by $ G$.  More precisely, we fix a character $ \chi : G \rightarrow \Cx $ and form the projective GIT quotient
$$ T^* N \ssslash_{0,\chi} G := \Phi^{-1}(0) \sslash_\chi G :=  \Proj \left( \bigoplus_{n \in \N} \C[\Phi^{-1}(0)]^{G, n\chi} \right) $$

 We have a natural morphism $ Y = T^*N \ssslash_{0,\chi} G \rightarrow X = T^* N \ssslash_{0,0}  G $.  In the physics literature, $Y $ is called the \textbf{Higgs branch} of the 3d supersymmetric gauge theory defined by $ G,N$.  $ G $ is called the \textbf{gauge group} and $ N $ is called the \textbf{matter}.
 
 The spaces $ Y, X $ carry Poisson structures coming from the usual symplectic structure on $ T^* N$.  This construction will not usually give a symplectic resolution; for example $ Y $ may not be smooth and $ Y \rightarrow X $ might not be birational (see Example \ref{eg:sl2}).   

There is a conical $ \Cx $ action on $ Y $ coming from its scaling action of $ T^* N $.  In order to define a Hamiltonian torus action, we need one piece of data.  We choose an extension $ 1 \rightarrow G \rightarrow \tG \rightarrow T \rightarrow 1 $, where $ T $ is the \textbf{flavour} torus, and an action of $ \tG $ on $ N $, extending the action of $ G $.  Then we obtain a residual Hamiltonian action of $ T $ on $ Y$ and $ X $.  In general, this action does not have finitely many fixed points.

\subsection{Hypertoric varieties} \label{se:hypertoric}
A very important class of symplectic resolutions are hypertoric varieties.  These are very combinatorial, so they are great examples for testing and formulating conjectures.  They were originally introduced by Bielawski-Dancer \cite{BD} from the perspective of hyperk\"ahler geometry, with the name ``toric hyperk\"ahler manifolds''.  Our approach will follow more closely the survey article of Proudfoot \cite{Nick}.

We work in the above setup of Hamiltonian reduction, with $G $ a torus, and $ \tG = (\Cx)^n $, so we have
\begin{equation} \label{eq:tori}
1 \rightarrow G \rightarrow (\Cx)^n \rightarrow T \rightarrow 1
\end{equation}
This gives us a linear action of $ G $ on $T^* \C^n = \C^n \oplus (\C^n)^*$.  We have a moment map $ \Phi : T^* \C^n \rightarrow \C^n \rightarrow \fg^* $ where the first map is given by $ (z_i, w_i) \mapsto (z_iw_i)$.

We fix a generic character $ \chi : G \rightarrow \Cx $ and consider $
Y = \Phi^{-1}(0) \sslash_\chi G $ and $ X = \Phi^{-1}(0) \sslash G $ as above.

The map $ (\Cx)^n \rightarrow T $ is equivalent to a map $ \phi : \Z^n \rightarrow \ft_\Z $ between the coweight lattices of these tori.  We assume that $\phi$ is \textbf{unimodular} (this means if we choose a matrix representing $ \phi$, then every invertible square submatrix has determinant $\pm 1$).  This ensures that $ Y $ is smooth. 

We also assume that the map $ G \rightarrow (\Cx)^n $ is not contained in any coordinate subtorus.  This is needed to ensure that the natural $ \Cx $ action on $ X $ is conical.

We have residual actions of $ T $ on $ Y$ and $ X $ with moment map $ X \rightarrow \ft^*$.

The structure of $ Y $ can be visualized by means of a hyperplane arrangement.  We consider the real vector space $ \ft_\R^* $ and define affine hyperplanes $ H_1, \dots, H_n$ by
$$
H_i := \{ v \in \ft_\R^* : \langle \phi(e_i) , v \rangle = \langle e_i, \tilde{\chi} \rangle \}
$$
where $ \tilde{\chi} \in (\R^n)^* $ is a lift of $ \chi \in \fg_\R^* $. 

These hyperplanes partition $ \ft_\R^* $ into chambers, and $ Y $ contains all the toric varieties associated to these chambers.  In particular, the core $ F_0 $ is the union of the toric varieties associated to the compact chambers.

\begin{Example}
	Our first example of (\ref{eq:tori}) is  
	\begin{gather} \label{eq:exampletori1}
	1 \rightarrow \Cx \xrightarrow{t \mapsto (t, \dots, t)} (\Cx)^n \rightarrow (\Cx)^n/\Cx \rightarrow 1
	\end{gather}
We choose $ \chi : \Cx \rightarrow \Cx $ to be the identity.  The resulting hypertoric variety is $ Y = T^* \PP^{n-1} $.  The hyperplane arrangement lies in $ \ft_\R^* = \{(a_1, \dots, a_n) : a_1 + \dots + a_n = 0 \} $ with affine hyperplanes
$$
H_i = \{ (a_1, \dots, a_n) : a_i = 1 \}
$$
There is one compact chamber given by $ \{(a_1, \dots,a_n) \in \ft_\R^* : a_i \le 1 \} $ which is an $ n$-simplex (whose associated toric variety is $ \PP^{n-1}$).
 
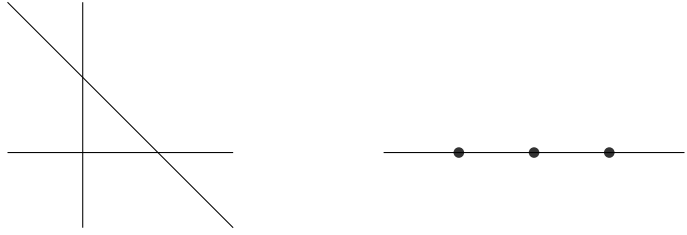
\begin{figure} 
\begin{tikzpicture}
	\draw (-1,0) -- (2,0);
	\draw (0,-1) -- (0,2);
	\draw (2,-1) -- (-1,2);
	\draw (4,0) -- (8,0);
	\fill [black,opacity=.8] (5,0) circle (2pt);
		\fill [black,opacity=.8] (6,0) circle (2pt);
			\fill [black,opacity=.8] (7,0) circle (2pt);
		
\end{tikzpicture}
\caption{The hyperplane arrangements for $ T^* \PP^3 $ and $ \widetilde{\C^2 / \Z_3}$.} \label{fig:hyp}
\end{figure}

A second example is 
\begin{gather} \label{eq:exampletori2}
	1 \rightarrow \{(t_1, \dots, t_n) : t_1 \cdots t_n = 1 \} \rightarrow (\Cx)^n \rightarrow \Cx \rightarrow 1 
	\end{gather}
and we choose $ \chi(t_1, \dots, t_n) = t_i^i $.  The resulting hypertoric variety is $ Y = \widetilde{\C^2 / \Z_n} $.  The hyperplane arrangement is the arrangment in $ \R $ given by the $n $ points $ 1, \dots, n  $.  Note that there are $ n-1 $ compact chambers: the intervals $ [1,2], \dots, [n-1,n] $, corresponding to the $ n-1 $ $\PP^1$s in the core.
\end{Example}

The short exact sequences of tori in (\ref{eq:exampletori1}) and (\ref{eq:exampletori2}) are dual to each other.  More generally, given 
\begin{equation*} 
	1 \rightarrow G \rightarrow (\Cx)^n \rightarrow T \rightarrow 1
\end{equation*}
we can dualize to
$$
1 \rightarrow T^\vee \rightarrow (\Cx)^n \rightarrow G^\vee \rightarrow 1
$$
where $T^\vee $ is the dual torus which satisfies $ \Hom(\Cx, T^\vee) = \Hom(T, \Cx) $.

The resulting pair $ Y, Y^\vee $ of hypertoric varieties are called \textbf{symplectic dual}.  This will be a special case of the general phenomenon that we will study in section \ref{se:SD}.  

\subsection{Quiver varieties} \label{se:quiver}

Another special case of the Hamiltonian reduction construction are Nakajima quiver varieties, where $ G, N $ come from the combinatorial data of a framed quiver.  These were first introduced by Nakajima \cite{Nak94}.

We fix a finite directed graph $Q = (I,E)$ (loops and multiple edges are allowed), with head and tail maps $ h,t : E \rightarrow I $.  Also, we fix two dimension vectors $ \Bv, \Bw \in \N^I $.  For $ i \in I$, let $ V_i = \C^{v_i}, W_i = \C^{w_i}$ and consider the space of  representations of the quiver $ Q $ on the vector space $ \oplus V_i$ framed by $ \oplus W_i$.
$$
N = \bigoplus_{e \in E} \Hom(V_{t(e)}, V_{h(e)}) \oplus \bigoplus_{i \in I} \Hom(V_i, W_i)
$$
This big vector space $ N $ has a natural action of $ G = \prod_i GL(V_i) $.  We form the cotangent bundle $T^* N $ and the take the Hamiltonian reduction by the action of $ G $ as in section \ref{se:Higgs}.  The resulting space $ Y = \Phi^{-1}(0) \sslash_\chi G $ is called a Nakajima quiver variety.  Here we choose $ \chi : G \rightarrow \Cx $ to be given by the product of the determinants.

On $ Y $, we have a Hamiltonian action of $ T = \prod_i (\Cx)^{w_i} $ inherited from its action on $ \oplus W_i $.  (In other words, we take $ \tG = G \times T $.)

The space $ Y $ is always smooth \cite[Cor 3.12]{Nak98}.  However $ \pi : Y \rightarrow X $ is not always birational.  Also, the Hamiltonian torus action does not always have finitely many fixed points.

\begin{Example} \label{eg:Anquiver}
	The following example is key.  Consider a linearly oriented type $ A_{n-1} $ quiver with $ \Bv = (1, \dots, n-1) $ and $ \Bw = (0, \dots, 0, n) $.
	\begin{gather*}
		\begin{tikzpicture}[scale=0.6]
			\draw  (0,0) circle [radius=1];
			\draw (4,0) circle[radius=1];
			\draw (12,0) circle [radius=1];
			\draw [->] (1.2,0) -- (2.8,0);
			\draw [->] (5.2,0) -- (6.8,0);
			\draw [->] (9.2,0) -- (10.8,0);
			\draw [->] (13.2,0) -- (14.8,0);
			\node at (0,0) {$V_1$};
			\node at (4,0) {$V_2$};
			\node at (12,0) {$V_{n-1}$};
			\node at (8,0) {$\cdots$};
			\draw (15,-1) rectangle (17,1);
			\node at (16,0) {$\C^n$};
		\end{tikzpicture} \\
	 N = \Hom(\C, \C^2) \oplus \cdots \oplus \Hom(\C^{n-1}, \C^n), \quad G = GL_1 \times \cdots \times GL_{n-1} 
	\end{gather*}
	The resulting quiver variety is $ Y = T^* Fl_n $, with $ X = \mathcal N_{\mathfrak{sl}_n}$.  
\end{Example}

\begin{Example} \label{se:instantons}
	Another important example is a quiver with one vertex and one self loop with $ V = \C^n $ and $ W = \C^r $.
\begin{gather*}			
	\begin{tikzpicture}[scale=0.5]
		\draw  (0,0) circle [radius=1];
		\draw [->] (1.2,0) -- (2.8,0);
		\draw [->] (-0.9,-0.7) .. controls (-3,-1.8) and (-3,1.8) .. (-0.9,0.7);
		\node at (0,0) {$\C^n$};
		\draw (3,-1) rectangle (5,1);
		\node at (4,0) {$\C^r$};
	\end{tikzpicture}\\
	N = \End(\C^n) \oplus \Hom(\C^n, \C^r),  \quad G = GL_n
	\end{gather*}
	
	In this case, $ Y = Y(r,n)$ is the moduli space of rank $ r $, torsion-free sheaves on $ \mathbb P^2 $, framed at $ \infty $ with second Chern class $ n$ (see \cite{Nakbook}).  In particular, $Y(1,n)$ is the Hilbert scheme of $n $ points in $ \C^2$.  ($Y(r,n) $ is also known as the Uhlenbeck compactification of the moduli space of $ SU(r) $ instantons on $ \C^2$.)
\end{Example}
	
\subsection{Affine Grassmannian slices} There is another important class of symplectic resolutions consisting of resolutions of slices in the affine Grassmannian, but we will postpone their discussion until section \ref{se:AGslices}.

\section{Topology of symplectic resolutions} \label{se:top}

\subsection{Symplectic leaves and the Springer sheaf} \label{se:leaves}
We will now study the topology of the map $ \pi $.

\begin{Theorem}
	Let $\pi : Y \rightarrow X $ be a symplectic resolution.
	\begin{enumerate}
		\item $ X $ has a finite partition $ X = \cup_{j \in J} X_j $ where each $ X_j $ is locally closed, smooth, and symplectic.  (These are the \textbf{symplectic leaves} of $ X $.)
		\item The map $ \pi $ is semismall.
		\item The pushforward $ \pi_* \C_Y $ is constructible with respect to the stratification by the symplectic leaves.
	\end{enumerate}
\end{Theorem}
The first and second parts are due to Kaledin \cite{Ka}.  The third statement is due to McGerty-Nevins \cite[Cor 4.19]{MN}.  One might imagine that the third statement is proven by showing that $ \pi $ is smooth over each $ X_j $.  However, the proof of McGerty-Nevins is more indirect and uses nearby cycles.

From the theorem, the \textbf{Springer sheaf} $ \pi_* \C_Y[2d] $ is perverse, where $ 2d = \dim Y$.  By the decomposition theorem, it is a direct sum of simple perverse sheaves associated to local systems on the strata $ X_j $.  To simplify the exposition in what follows, we will assume that only trivial local systems appear in this decomposition.  For ADE quiver varieties, this was proven by Nakajima \cite[Prop 15.32]{NaQA}, while for hypertoric varieties it was proven by Proudfoot-Webster \cite[Prop 5.2]{PW}.  (This assumption does not hold for $ T^* G/B $, if $ G $ is not $ SL_n $.)

Fix a point $ x_j \in X_j $ in each stratum and let $ F_j = \pi^{-1}(\{x_j\}) $.  Under our assumption from Remark \ref{re:nodeg1}, $ X_0 = \{0 \} $ is a symplectic leaf and $ F_0 = \pi^{-1}(\{0\}) $ is the core.

From the semismallness of $ \pi $, we know that $ \dim F_j \le \frac{1}{2} \codim X_j $.  We will assume that we actually have equality, since this often holds in examples (if there are strata where the fibre has smaller dimension, then these will not contribute to the decomposition below).  Thus from the Beilinson-Bernstein-Deligne decomposition theorem, we obtain
  \begin{equation} \label{eq:BBD}
  \pi_* \C_Y[2d] = \bigoplus_{j \in J} \IC_{X_j} \otimes H(F_j)
  \end{equation}
  where for any algebraic variety $ X $, $ H(X) := H_{2 \dim X}^{BM}(X)$ is the top Borel-Moore homology, and $ \IC_X $ is the intersection cohomology sheaf.

 \begin{Example} \label{eg:strataSpringer}
 	Consider $ T^* Fl_n \rightarrow \cN $.  The symplectic leaves of $ \cN $ are the nilpotent orbits $ X_\tau $, where are indexed by the partitions of $ n $. The fibre $ F_\tau $ is called a Springer fibre.  Its irreducible components are classified by standard Young tableaux of shape $ \tau$.     
 \end{Example}

\begin{Remark}
	Suppose that $ G $ is a semisimple group, and $ Y = T^* G/B $ is the cotangent bundle of its flag variety.  Springer theory shows that the sheaf $ \pi_* \C_Y[2d] $ carries an action of the Weyl group $ W $.  McGerty-Nevins \cite{MN} generalized this construction to any conical symplectic resolution, by defining an action of the Namikawa Weyl group (see section \ref{se:deform}) on the Springer sheaf.
\end{Remark}

\subsection{Hyperbolic decomposition}
  Assume that $ Y \rightarrow X $ carries a Hamiltonian $T$-action with finitely many fixed points.  Choose a generic $ \Cx \subset T $ yielding an action of $ \Cx $ on $ Y $ with the same fixed points as $ T $.
  
  This choice of $ \Cx $ allows us to define attracting sets
  $$Y^+ := \{ y \in Y : \lim_{s \to 0} s \cdot y \in Y^T \}, \quad X^+ := \{ x \in X: \lim_{s \to 0} s \cdot x \in X^T \} $$
  By equivariance, we have $ Y^+ \rightarrow X^+ $.  We define $ X_j^+ = X_j \cap X^+$.  Because the action is Hamiltonian and $ Y $ is smooth, we have $ \dim Y^+ = d $.  This implies $ \dim X^+ \le d $.  (It should be possible to prove that $ \dim X^+ = d $ using that the deformation $ \tX $ (see section \ref{se:deform}) is $ \Cx$-equivariant and that the general fibre is smooth.  However, we could not find an argument like this in the literature.)

 Let $ i_X : X^+ \rightarrow X $ be the inclusion.  We have the hyperbolic stalk functor
  \begin{align*}
  \Phi : D_c(X) &\rightarrow D(\Vect) \quad
  \cF \mapsto H^\bullet(X^+, i_X^! \cF)
  \end{align*}
 where $ D_c(X) $ denotes the derived category of sheaves on $ X $ constructible with respect to the stratification $ \{ X_j \} $.
 
 We assume that each $ \overline{X_j} $ admits a symplectic resolution such that the Hamiltonian $ T $ action extends to this resolution (this can often be checked in examples).  This implies that $ \dim X_j^+ \le \frac{1}{2}\dim X_j $ and so the inclusion of $ X_+ $ into $ X $ is hyperbolic semismall in the sense of \cite[Def 5.5.1]{NakPCMI}. (Alternatively, if we know that $ \dim F_j = \frac{1}{2}\codim X_j $ for all $j$, we can also deduce this inequality.)  Thus Theorem 5.5.3 of \cite{NakPCMI} applies and we have the following.  
  
  \begin{Theorem} \label{th:hypess}
	For any perverse sheaf $ \cF$, $ \Phi(\cF) $ is concentrated in degree 0.  In particular, $ \Phi(\IC_{X_j}) = H(\overline{X_j^+}) $.
  \end{Theorem}

  Considering the diagram
\begin{equation*}
	\begin{tikzcd}
		Y^+ \arrow{r}{i_Y} \arrow{d}{\pi^+} & Y \arrow{d}{\pi} \\
		X^+ \arrow{r}{i_X} & X 
	\end{tikzcd}
\end{equation*}
we see that 
\begin{gather*} \Phi(\pi_* \C_Y[2d]) = H^\bullet(X^+, i_X^! \pi_* \omega_Y[-2d]) \\
	= H^\bullet(X^+, \pi^+_* i_Y^! \omega_Y[-2d]) = H^\bullet(Y^+,\omega_{Y^+}[-2d])
\end{gather*}
From Theorem \ref{th:hypess}, we conclude that $ H^\bullet(Y^+,\omega_{Y^+}[-2d]) $ is concentrated in degree 0, and thus equals  $ H(Y^+)$.

Thus applying the functor $ \Phi $ to the decomposition (\ref{eq:BBD}) yields the following, which we call the \textbf{hyperbolic BBD decomposition}
\begin{equation} \label{eq:fund}
	H(Y^+) = \bigoplus_{j \in J} H(\overline{X_j^+}) \otimes H(F_j) = H(F_0) \oplus \bigoplus_{j \ne 0, top}  H(\overline{X_j^+}) \otimes H(F_j) \oplus H(X^+) 
\end{equation}
  In this decomposition, we highlight the two extremal pieces, corresponding to the point stratum $ j = 0 $, which contributes $ H(F_0)$  and the open stratum $ j = top $ which contributes $ H(X^+) $.

Note that the total dimension of $ H(Y^+) $ is the same as the total dimension of $ H^\bullet(Y) $, as both equal the number of fixed points $ Y^T $.  So we view the hyperbolic BBD decomposition as a replacement for the usual decomposition of $ H^\bullet(Y)$.

\begin{Example} \label{eg:fund}
	Consider $ \widetilde{\C^2/\Z_n} \rightarrow \C^2/\Z_n$.  There are two strata $ X_0 = \{0\} $ and $ X_1 = X \setminus X_0 $.  We can see that $ F_1$ is a point and $ F_0 $ is a union of $ n-1$ projective lines. 
	
	We have $ X^+ = \C \subset \C^2 / \Z_n $, the image of the $ x$-axis and then $ Y^+ = F_0 \cup X^+ $.  
	
	So the hyperbolic BBD decomposition gives
	$$
	H(Y^+) = H(\overline{X_0^+}) \otimes H(F_0) \oplus H(\overline{X_1^+}) \otimes H(F_1)
	$$
	which has dimensions $
	n = 1 \cdot (n-1) + 1 \cdot 1  $.
	
	On the other hand, consider $ T^* \PP^{n-1} \rightarrow \{ A^2 = 0, \rk A \le 1 \} $.  In this case, we have two strata $ X_0 = \{0\} $ and $ X_1 = X \setminus X_0 $.  We have $F_1 = \{pt\}$ and $F_0 = \PP^{n-1} $. 
	
	For the attracting sets, we see $ X^+ = X \cap \mathfrak n $.  This space has $ n-1 $ irreducible components (see \cite[Example 3.3.15]{CG}).

		So the hyperbolic BBD decomposition gives
	$$
	H(Y^+) = H(\overline{X_0^+}) \otimes H(F_0) \oplus H(\overline{X_1^+}) \otimes H(F_1)
	$$
	which has dimensions $ n = 1 \cdot 1 + (n-1) \cdot 1  $.
	
\end{Example}

\begin{Example}
	Continuing from Example \ref{eg:strataSpringer}, we consider $ X = \cN_{\gl_n} $ and $ Y = T^* Fl_n$.  Then $ X^+ = \mathfrak n $, the set of strictly upper-triangular matrices, and the irreducible components of $ \overline{X^+_\tau} $ are known as orbital varieties (here, as in Example \ref{eg:strataSpringer}, $\tau$ ranges over the partitions of $ n$). There is a bijection between the irreducible components of $ \overline{X^+_\tau} $ and the set of standard Young tableaux of shape $ \tau $. 
	
	On the other hand, $ \dim H(Y^+) = n!$, since $ Y^+ $ is the union of the conormal bundles to the Schubert cells in $ Fl_n$.  Thus the numerics of the hyperbolic BBD decomposition yields
$$
n! = \sum_{\tau \vdash n} t_\tau^2
$$
where $t_\tau $ is the number of standard Young tableaux of shape $ \tau$.  A combinatorial proof of this numerical identity is given by the Robinson-Schensted correspondence. In fact, this bijection is encoded in the combinatorics of the components of $ Y $ (see \cite[Remark 3.3.22]{CG}).
\end{Example}

\subsection{Transversal slices to symplectic leaves}
Consider a symplectic leaf $ X_j $ with basepoint $ x_j $.  A \textbf{transversal slice} to $ X_j $ is another conical symplectic singularity $ V_j $ such that we have a Poisson isomorphism of formal completions
\begin{equation} \label{eq:formalproduct}
\widehat{X}_{/x_j} \cong \widehat{V_j}_{/0} \hat{\times} \widehat{X_j}_{/x_j}
\end{equation}
Kaledin \cite{Ka} proved that such transversal slices always exists, except without the condition of being conical (see also \cite{Sch}).  In many examples, there exist natural constructions of these slices and they are all conical.

If we further assume that $ V_j $ admits a symplectic resolution $ \widetilde V_j$, then it is natural to expect that (\ref{eq:formalproduct}) lifts to the resolution and in particular that the core of $ \widetilde V_j \rightarrow V_j $ is isomorphic to $F_j $.

\subsection{Quantum cohomology} \label{se:QC}
Another object of intense study is the equivariant quantum cohomology of symplectic resolutions.  This was initiated in the work of Braverman-Maulik-Okounkov \cite{BMO}.

Let $QH^\bullet_{T\times\Cx}(Y)$ be the equivariant quantum cohomology ring of $Y$. The underlying graded vector space of $QH^\bullet_{T\times\Cx}(Y)$ is equal to the tensor product of $H^\bullet_{T\times\Cx}(Y)$ with the completion of the semigroup
ring of effective curve classes in $H_2(Y, \Z)$.  Let $\star$ denote the quantum product (shifted by the canonical theta characteristic)
and let $\hbar\in H^\bullet_{T\times\Cx}(pt)$ be the weight of the symplectic form.

In \cite[Section 2.3.4]{okounkov2015enumerative}, Okounkov conjectured that there exists a finite set $\Delta_+\subset H_2(Y, \Z)$,
and an element $L_\alpha\in H^{2d}(Y \times_X Y)$ for each $\alpha\in\Delta_+$, which control quantum multiplication by divisors.  Namely, for all $u\in H^2_{T \times \C^{\times}}(Y)$, 
$$u\star \cdot = 
u \cup \cdot + \hbar \sum_{\alpha\in\Delta_+} \langle \alpha, u\rangle \frac{q^\alpha}{1-q^\alpha} L_\alpha(\cdot),$$
where $L_\alpha$ acts via convolution. $\Delta_+$ is called the set of {\bf positive K\"ahler roots}, and $\Delta := \Delta_+ \cup -\Delta_+$ is called the set of {\bf K\"ahler roots}.

We use the quantum multiplication by divisors to define the \textbf{quantum connection}.  We define the \textbf{K\"ahler torus} $ K := H^2(Y, \C^\times)$ and consider the regular locus $ \Kreg := K \setminus \cup_{\alpha \in \Delta_+} \{ q^\alpha = 1 \} $.  Then we define a connection on the trivial vector bundle over $ \Kreg $ with fibre $ H^\bullet_{T \times \Cx}(Y) $ by letting $ u\in H^2(Y) = \mathfrak k$
act by the operator $ \partial_{u} + u \star $.

It is very interesting to try to identify the resulting connection.  Here are some results.

\begin{Theorem}
	\begin{enumerate}
		\item If $ Y = T^* G/B $, then the quantum connection is the affine Knizhnik-Zamolodchikov connection \cite{BMO}.
		\item If $ Y $ is a hypertoric variety, then the quantum connection is a Gelfand-Kapranov-Zelevinsky system \cite{McBreenShenfeld}.
		\item If $ Y $ is a quiver variety, then the quantum connection is the trigonometric Casimir connection \cite{MO}.
		\item If $ Y $ is the resolution of an affine Grassmannian slice, then the quantum connection is the trigonometric KZ connection \cite{Danilenko}.
	\end{enumerate}
\end{Theorem}

In order to prove these results, the authors use the deformation $ \tY$ described below, and also the notion of \textbf{stable basis}, which was developed by Maulik-Okounkov \cite{MO}.

These connections each have an interesting monodromy which gives a representation of $ \pi_1(\Kreg) $ on $H^\bullet_{T \times \Cx}(Y) $.  Bezrukavnikov-Okounkov \cite[Conjecture 1]{BezReal} have conjectured that this monodromy representation can be categorified to an action of $ \pi_1(\Kreg) $ on $ D^b Coh^{T \times \Cx}(Y)$.

\section{Deformations and Quantizations} \label{se:alg}

\subsection{Deformations} \label{se:deform}
Namikawa \cite{Nam08} studied deformations of conical symplectic resolutions, building on previous results of Kaledin and coauthors.

Let $ Y \rightarrow X $ be a conical symplectic resolution.  We will be interested in deformations of $ Y $ and $ X $ as Poisson varieties.
\begin{Theorem}  \label{th:tY}
	There is a universal Poisson deformation $ \tY \rightarrow H^2(Y) $, where $ Y$ is identified with the fibre $Y_0$.
\end{Theorem} 
At first glance, it looks mysterious that $H^2(Y) $ classifies Poisson deformations of  $Y$.  This is explained by the period map.  More precisely, all the fibres $Y_\theta$, for $  \theta \in H^2(Y) $, are symplectic with symplectic form $ \omega_\theta$. The fibration $ \tY \rightarrow H^2(Y)$ is  topologically trivial (see \cite[Lemma 6.4]{MN}). Thus, the cohomology of the fibres are identified and so we can consider $ [\omega_\theta] \in H^2(Y_\theta) = H^2(Y) $.  As a result of the construction of Namikawa, we have $ [\omega_\theta] = \theta $.

Similarly, we can study Poisson deformations of $ X$.  The base of the universal deformation is given by the quotient of $ H^2(Y) $ by a finite group $ W $, called the \textbf{Namikawa Weyl group}.  The following result is also due to Namikawa \cite{Nam08}.

\begin{Theorem}
	There is a universal Poisson deformation $ \tX \rightarrow H^2(Y)/W $ of $ X $.  Moreover these two universal deformations fit together into the following diagram
	\begin{equation} \label{eq:simult}
	\begin{tikzcd}
		\tY \arrow{r} \arrow{d} & H^2(Y) \arrow{d} \\
		\tX \arrow{r} & H^2(Y)/W 
	\end{tikzcd}
\end{equation}
\end{Theorem}
The conical $ \Cx $ actions on $ Y$ and $ X $ extend to these deformations, compatible with the scaling action on $H^2(Y)$.  The actions of the Hamiltonian torus on $Y $ and $ X $ extends to a fibrewise action on $ \tY $ and $\tX $ \cite[Lemma 6.4]{MN}.

\begin{Example}
	Let $ G $ be a semisimple group and let $ Y = T^* G/B $ and $ X = \mathcal N$.  Then the diagram (\ref{eq:simult}) is given by Grothendieck's simultaneous resolution
		\begin{equation*} 
		\begin{tikzcd}
			\tY = \widetilde{\fg} \arrow{r} \arrow{d} & \ft \arrow{d} \\
			\tX = \fg \arrow{r} & \ft/W 
		\end{tikzcd}
	\end{equation*}
where $ \widetilde{\fg} = \{ ([g], x) : x \in Ad_g(\mathfrak b) \} $.  The map $ \fg \rightarrow \ft/W $ is the adjoint quotient map (given by the characteristic polynomial for $\fg = \ssl_n$).
\end{Example}

\begin{Example} \label{eg:deformHiggs}
	Suppose that $ Y = T^*N \ssslash_{0,\chi} G $ is the Higgs branch of the gauge theory defined by $ G, N$.  Let $ B = \Hom(\fg, \C) \subset \fg^*$, the space of 1-dimensional representations of $ \fg $.
	
	We can build a deformation over the base $ B $ by $ \tY = \Phi^{-1}(B) \sslash_\chi G \rightarrow B $.

    We have an injective ``Kirwan map'' $ B \rightarrow H^2(Y) $ given by linearly extending $ \xi \mapsto c_1(L_\xi)$, where $ \xi \in \Hom(G, \Cx) $ and $ L_\xi $ is the resulting line bundle on $ Y$.  When this map is surjective, then $ \tY \rightarrow B $ is the universal deformation from Theorem \ref{th:tY}.
\end{Example}  

\subsection{Quantizations} \label{se:quantize}
Let $ Y \rightarrow X $ be a conical symplectic resolution.  Then $ \C[X] $ is a $\N$-graded Poisson algebra (where the grading comes from the conical $ \Cx $ action).  

Let $ A $ be a graded $ \C[\hbar]$-algebra with $ \deg \hbar = 2$ and let $ \bar A = A / \hbar A$.  Assume that $A $ is free as a $ \C[\bar]$-module and that $ \bar A $ is commutative.  Then $ \bar A $ acquires a Poisson structure defined by
$$
\{ \overline{a}, \overline{b} \} = \overline{\tfrac{1}{\hbar}(a b - ba)} \quad \text{ for all $a, b \in A$}
$$
$ A $ is called a \textbf{graded quantization} of $ X $ if we are given an isomorphism of graded Poisson algebras $ \bar A \cong \C[X] $.

Since $ X $ is the affinization of $ Y $, $ \C[X] $ is the global sections of the structure sheaf $ \cO_Y $, which is a sheaf of Poisson algebras.  A \textbf{quantization} of $ Y $ is a sheaf $ \cA $ of $ \C[\hbar] $-algebras on $ Y $, such that $\bar \cA \cong \cO_Y $ as sheaves of Poisson algebras.  In this case $ \Gamma(Y, \cA) $ is a quantization of $ X $.  (Since $\cA $ is a sheaf, it doesn't make sense to ask that $ \cA $ is graded, instead we assume that it carries a $ \Cx$-equivariant structure, where $ \Cx $ is the conical action on $ Y $.)

The following result is due to Losev \cite[Cor 2.3.3]{Losev} building on previous work of Bezrukavnikov-Kaledin \cite{BeKa}.
\begin{Theorem}
The universal Poisson deformation $ \tY \rightarrow H^2(Y) $ admits a unique quantization $ \cA $. 
\end{Theorem}
Taking global sections $ \A :=\Gamma(\tY, \cA) $ gives a quantization of $ \tX $, containing $ \C[H^2(Y)] $ as a central subalgebra.  Each $ \theta \in H^2(Y) $ gives a quantization $ A_\theta := \A \otimes_{\C[H^2(Y)]} \C $ of $ X $.  We have $ A_\theta \cong A_{w\theta} $ for $ w \in W $, and so we regard $ H^2(Y)/W $ as the quantization parameters for $ X $.

\begin{Example}
	Once again, $ Y = T^* G/B $ gives a good example. A quantization of $ Y $ is provided by the sheaf of microdifferential operators on $ G/B $.  The global sections of this sheaf is given by the global differential operators on $ G/B $, which in turn is isomorphic to $ U_\hbar \fg $ (the asymptotic enveloping algebra) modulo the positive part of the centre. More generally, we identify $ H^2(Y) = \ft^* $.  Then each $ \theta \in \ft^* $ gives us the quantization $A _\theta = U_\hbar \fg / Z_\theta $, where $ Z_\theta $ is the ideal generated by $ z - \theta(z) $ for $ z\in Z(U\fg) = (\Sym \ft)^W$.  In this situation, we have $ \A = \tilde U_\hbar \fg := U_\hbar \fg \otimes_{Z(U \fg)} \Sym \ft $.
\end{Example}

\begin{Example}
	In the Hamiltonian reduction setting, where $ X = T^*N \ssslash_{0,0} G $, we can construct quantizations as follows.  We begin with the quantization of $ T^* N$, which is given by the algebra $ D(N)$ of differential operators on $ N$.  Then we apply the procedure of quantum Hamiltonian reduction (see \cite[Section 3.4]{BPW}).  The quantization parameter $ \theta $ comes from the moment map level in the quantum Hamiltonian reduction.
\end{Example}

\subsection{Category $\cO$}
Fix $ \theta \in H^2(Y)/W $.  We will be interested in categories of modules over $ A_\theta $, especially the generalization of the Bernstein-Bernstein-Gelfand category $ \cO $.

Let $ Y \rightarrow X $ be a conical symplectic resolution, equipped with a Hamiltonian $ T $ action.  The $ T $-action on $ \C[X] $ quantizes to a $ T $-action on $ A_\theta $.  The derivative of this action is inner, since the moment map quantizes to a map $ \ft \rightarrow A_\theta $.

 Fix a generic $ \rho : \Cx \rightarrow T $.  Using $ \rho $, we can restrict the $ T $ action on $ A_\theta$ to a $ \Cx $ action, which gives a $ \Z$-grading on $ \C[X] $ and $ A_\theta $.

\begin{Example}
	Take $ Y = T^* G/B $ and so $ A_\theta = U_\hbar \fg / Z_\theta $.  Then the Hamiltonian torus action comes from the maximal torus $ T $ of $ G $ and we choose the usual $ \rho $.  With this choice, the $ \Z $ grading on $ U_\hbar \fg / Z_\theta $ assigns $ \deg E_i = 1, \deg F_i = -1 $ where $ E_i, F_i $ are the usual Chevalley generators.
\end{Example}

Let $ \C[X]^+, A_\theta^+ $ be the positively graded parts with respect to $ \rho $.  There is a compatibility between $ \C[X]^+ $ and $X^+ $ as follows; see \cite[1.3.4]{Dr} for a scheme-theoretic version.
\begin{Lemma} \label{le:idealplus}
	The radical ideal generated by $ \C[X]^+ $ defines $ X^+ $ as a subvariety of $ X $.
\end{Lemma}

This motivates the following category.
\begin{Definition}
	The category $ \cO $ for $ A_\theta $, denoted $ \Oof{A_\theta}$, is the full subcategory of finitely-generated $ A_\theta$ modules on which $ A_\theta^+ $ acts locally nilpotently.
\end{Definition}

 Braden-Proudfoot-Licata-Webster constructed a characteristic cycle map 
 \begin{equation} \label{eq:ccmap}
 	cc : K(\Oof{A_\theta}) \rightarrow H(Y^+)
 \end{equation}
as follows.  We begin with a $A_\theta $-module $M $ in category $ \cO $.  Generalizing the work of Beilinson-Bernstein \cite{BB}, we localize $M $ to a sheaf $ \cM $ of $ \cA_\theta$-modules on $ Y $.  Then we take $ \cM / \hbar \cM $ which will be a coherent sheaf on $ Y $.  Since the positively graded part of $ A_\theta $ acts locally nilpotently, by Lemma \ref{le:idealplus} this coherent sheaf will be set-theoretically supported on $ Y^+ $.  Thus its support defines a cycle in $H(Y^+)$.  

The characteristic cycle map is often, but not always, an isomorphism.  In particular, \cite[Theorem 6.5]{BLPW} proves that the related map $ K(\Oof{\cA_\theta}) \rightarrow H(Y^+) $ is an isomorphism.  This implies that  (\ref{eq:ccmap}) is an isomorphism whenever the localization functor gives an equivalence of derived categories, which is true for almost all $ \theta$.

\section{Symplectic duality} \label{se:SD}

\subsection{Background and some examples} \label{se:3dmirror}
Symplectic resolutions come in dual pairs $ Y \rightarrow X, Y^! \rightarrow X^! $ with certain structures for $ Y $ matching other structures for $ Y^! $. In mathematics, symplectic duality was first discovered by Braden-Proudfoot-Licata-Webster \cite{BLPW}, by generalizing Gale dual hypertoric varieties, and the duality between moduli spaces of instantons.

 However, these symplectic dual pairs had already been discovered in physics by Intriligator-Seiberg \cite{IS}, related to mirror duality for 3d supersymmetric quantum field theories (for this reason, another name for symplectic duality is \textbf{3d mirror symmetry}).  More precisely, every such field theory has two moduli spaces called the Higgs branch and the Coulomb branch.  These spaces are often symplectic resolutions.  Mirror duality exchanges these two spaces; that is, the Higgs branch of one theory is isomorphic to the Coulomb branch of the dual theory.  On the other hand, for a given theory, the Higgs and Coulomb branches are symplectic dual pairs in the mathematical sense.
 
Here are some examples of symplectic dual pairs.
\begin{center}
	\begin{tabular}{ c|c }
		$Y$ & $Y^!$ \\
	\hline
$T^* \PP^{n-1}$ & $\widetilde{\C^2/ \Z_n}$ \\
$T^* G/B$ & $T^* G^\vee / B^\vee $\\
hypertoric variety & dual hypertoric variety	\\
 ADE quiver varieties & affine Grassmannian slices \\
$Y(r,n)$ from Example \ref{se:instantons} & $\operatorname{Hilb}_n(\widetilde{\C^2/\Z_r})$ \\
\end{tabular}
\end{center}

We would like to make a statement on the ontological nature of symplectic duality.  We do not formulate a precise definition of symplectic duality.  Instead, we will list statements (\ref{eq:sdT}), (\ref{eq:sdS}), (\ref{eq:sdC}), (\ref{eq:sdO}), (\ref{eq:sdH}), (\ref{eq:sdR}) which are known or expected in the examples above and for other dual pairs.  In these statements, a certain structure on $ Y $ matches a different structure on $ Y^!$.  Moreover, we expect that symplectic duality is ``involutive'', so that $ (Y^!)^! = Y $.  Thus, each of these statements below implies another statement in the reverse direction.

\begin{figure}
 \includegraphics[trim=0 280 0 60, clip, width=0.6\textwidth, angle=0]{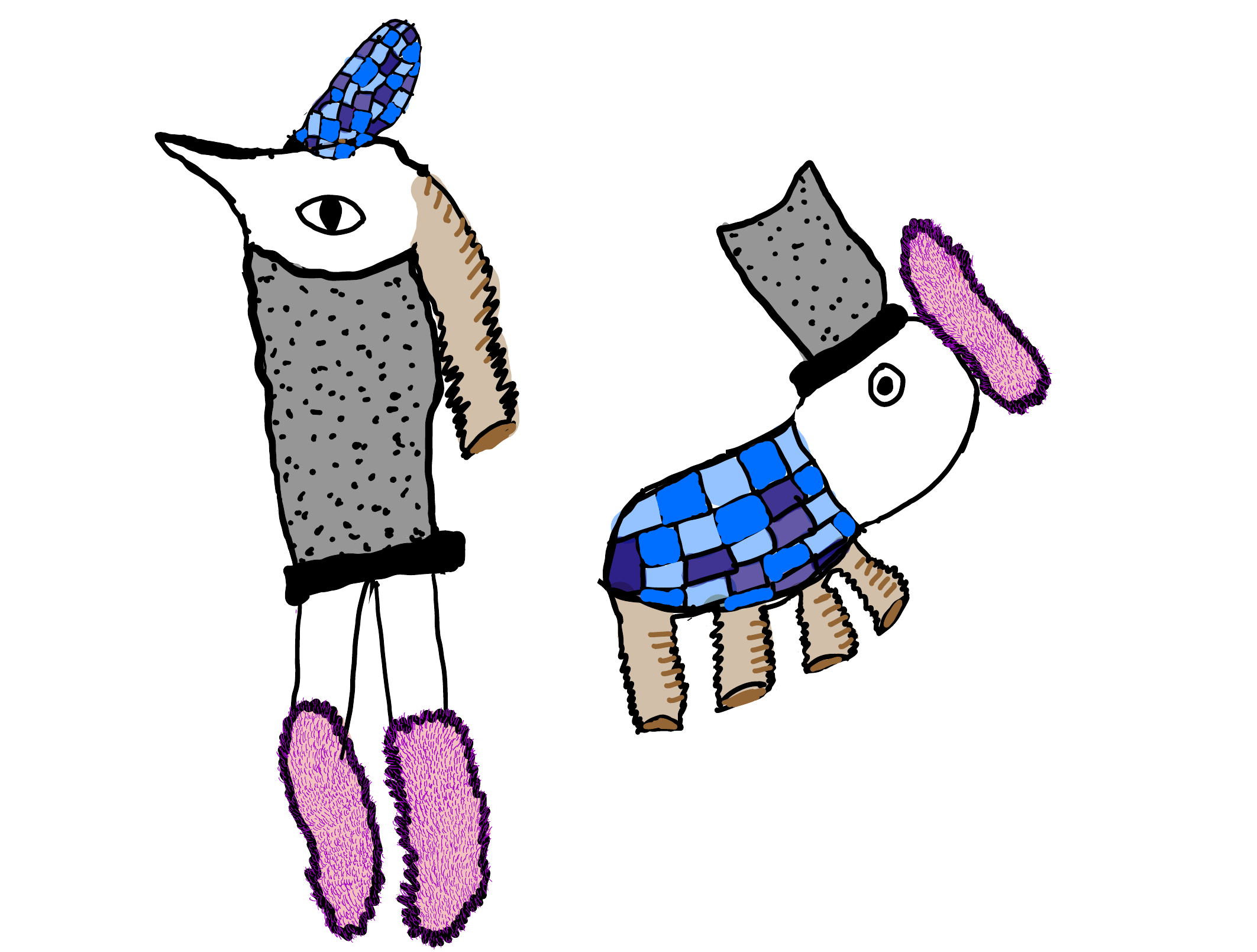} 
\caption{An artist's view of symplectic duality: certain structures on one creature match different structures on the other creature.}
\end{figure}

The reader should keep in mind that these statements are made under the ``best possible conditions'', when both spaces are conical symplectic resolutions with Hamiltonian torus actions with finitely many fixed points.  We also continue with the assumption that only trivial local systems appear in the decomposition of the Springer sheaf.  Finally, we assume that all symplectic leaves are \textbf{special} in the sense of \cite[section 6.3]{BLPW}.  These assumptions hold for hypertoric varieties and the ADE quiver varieties considered in section \ref{se:qv2}. (If some of these assumptions do not hold, then the statements below require modification.)

The collection of statements below is based on the foundational paper \cite{BLPW} and discoveries since then.  We emphasize that this is not meant to be an exhausive list; there are other matchings structures in the literature and the reader is free to discover more.

In what follows, $ Y \rightarrow X $ and $ Y^! \rightarrow X^! $ denote dual symplectic resolutions.  All spaces/structures associated to $ Y^! $ will be marked with a $ ! $.

\subsection{Hamiltonian torus and second cohomology}
We assume that $ Y \rightarrow X $ is equipped with the action of a Hamiltonian torus $ T $.  Symplectic duality identifies the Lie algebra $ \ft $ of this torus with the second cohomology group $ H^2(Y^!) $.

In fact, both of these vector spaces possess additional structures.  First, each carry natural $ \Z$-lattices.  For $ \ft $, this is the coweight lattice $ \ft_\Z$, and for $ H^2(Y) $, this is $ H^2(Y, \Z) $.  Second, the real form of each contains a cone.  In $ \ft_\Z $, we consider the set of all $ \rho : \Cx \rightarrow T  $ whose attracting sets give $Y^+$; this is the set of integral points of a cone in $ \ft_\R $, called the \textbf{attracting cone}.  In $ H^2(Y,\R)$ we have the \textbf{ample cone} of $ Y$ (by Kaledin \cite{Ka}, we have $ H^2(Y,\Z) = \operatorname{Pic}(Y) $).

\begin{equation} 
	\begin{gathered}\label{eq:sdT} \tag{\textbf{T}}
	\text{There is an isomorphism $ \ft_\Z \cong H^2(Y^!,\Z) $,} \\ \text{taking the attracting cone to the ample cone.}
	\end{gathered}
\end{equation}
In particular, $ T $ is identified with the K\"ahler torus $ K^! $ (defined in section \ref{se:QC}).

\subsection{Matching of strata}
Recall that from section \ref{se:leaves}, we have a stratification of $ X $ by symplectic leaves $ \{ X_j \}_{j \in J } $.  The set $ J $ of leaves has the structure of a poset by containment.  For each leaf, we will assume that we have transversal slices $ V_j $ which themselves admit symplectic resolutions $ \widetilde V_j $.

On the dual side, we also have leaves $ X^!_{j} $ for $j \in J^! $, which we will also assume admit symplectic resolution $ Y^!_{j}$.
\begin{equation}
	\begin{gathered} \label{eq:sdS} \tag{\textbf{S}}
	\text{There is an order reversing bijection $ J \rightarrow J^! $ denoted $ j \mapsto j^! $} \\
	\text{ such that $ \widetilde V_j $ and $ Y^!_{j^!} $ are symplectic dual}
	\end{gathered}
\end{equation}
Many examples of these posets of strata and these order reversing bijections are given in \cite{GN}.

\subsection{Matching of cycles}
Recall that we have the hyperbolic BBD decompositions  (\ref{eq:fund})
\begin{equation*} 
	H(Y^+) = \bigoplus_{j \in J} H(\overline{X_j^+}) \otimes H(F_j) \quad 	H({Y^!}^+) = \bigoplus_{j \in J^!} H(\overline{{X^!}_j^+}) \otimes H({F^!}_j)
\end{equation*}
Under symplectic duality, $H(Y^+) \cong H({Y^!}^+) $.  Assume that (\ref{eq:sdS}) holds, so that these two decompositions are indexed by the same set but in reversed order.  We also want these decompositions to be reversed in a second sense, exchanging fibres and attracting sets of strata.

\begin{equation}
\begin{gathered} \label{eq:sdC} \tag{\textbf{C}}
	\text{ There are bijections $ \Irr \overline{X_j^+} \cong \Irr {F^!}_j $, $ \Irr F_j \cong \Irr \overline{{X^!}_j^+} $ }
\end{gathered}
\end{equation}

Because of (\ref{eq:sdS}), it suffices to construct bijections $ \Irr X^+ \cong \Irr F^!_0 $ and $ \Irr F_0 \cong \Irr {X^!}^+ $.

\subsection{Koszul duality for Category $\cO$}
Let $ A = A_\theta $ be a quantization of $ Y \rightarrow X $ where $ \theta $ is generic and integral.  Because of (\ref{eq:ccmap}), we think of $ \Oof{A} $ as a categorification of $ H(Y^+) $.  From (\ref{eq:sdC}), we expect an isomorphism $ H(Y^+) \cong H({Y^!}^+) $.  Thus it is reasonable to expect an equivalence between the category $ \cO $ for dual pairs. Motivated by the work of Beilinson-Ginzburg-Soergel \cite{BGS} on Kozsul duality for category $ \cO $ for a semisimple Lie algebra, Braden-Licata-Proudfoot-Webster \cite{BLPW} conjectured the following.
\begin{gather} \label{eq:sdO} \tag{\textbf{O}}
	\text{There is a Koszul duality between $ \Oof{A} $ and $ \Oof{A^!}$} 
\end{gather}
This means that we have graded lifts $ \tilde \cO $ and $ \tilde \cO^! $ of these two categories and equivalence between their derived categories $ D(\tilde \cO) \cong D(\tilde \cO^!) $ of a specified form.

\subsection{Hikita conjecture}
The following was first observed by Hikita \cite{Hikita} and then extended to the non-commutative setting by Nakajima.

Let $ X $ be a scheme with an action of $ \Cx $.  Then it has a fixed point subscheme $ X^{\Cx} $ (see \cite[Def 1.2.1]{Dr}), which can have an interesting non-reduced structure, even if $ X^{\Cx} $ is a single point.  Suppose that $ X = \Spec A $ is an affine scheme.  Then the $\Cx $ action on $ X $ corresponds to a $ \Z$-grading on $ A $, and we have $X^{\Cx} = \Spec B(A)$ (see \cite[Ex 1.2.3]{Dr}), where $ B(A) $ is defined as follows.
\begin{Definition}
If $ A = \displaystyle{\mathop{\oplus}_{k \in \Z}} A(k) $ is a $ \Z$-graded algebra, then we define the \textbf{$B$-algebra} of $ A$, by
$$ B(A) = A(0) / \sum_{k > 0} A(-k) A(k)
$$
Note that $ \sum_{k > 0} A(-k) A(k) $ is a two-sided ideal of $ A(0)$.
\end{Definition}

On the other hand, if $ A $ is a non-commutative $ \Z$-graded algebra, then $ B(A) $ is closely related to the study of category $ \cO $ for $ A$.  To be more precise, we have a functor $ \Oof{A} \rightarrow \Mof{B(A)} $ taking $ M $ to $ M^{A^+} = \{ m \in M : a m = 0 \text{ for all $a \in A^+$}\} $.  Moreover, we have an adjoint functor of induction, which can be used to define standard objects in category $ \cO$ (see \cite[\S 5]{BLPW}).

Let $ Y \rightarrow X $ be a symplectic resolution and as before fix some generic $ \rho : \Cx \rightarrow T $.  This gives us a $ \Z$-grading on $ \C[X] $ and on the universal quantization $\A$.  We consider the resulting $B$-algebra $B(\A)$.  Note that the map $ \C[H^2(Y)][\hbar] \rightarrow \A $ lands in $  \A(0) $ and so descends to a map $ \C[H^2(Y)][\hbar] \rightarrow B(\A) $.  Since $ X^T $ is a single point, $ B( \A) $ is finite (as a module) over $ \C[H^2(Y)][\hbar]$.

On the symplectic dual side, we will be interested in $ H^\bullet_{T^! \times \Cx}(Y^!) $ which is finite (as a module) over $ H^\bullet_{T^! \times \Cx}(pt) = \C[\ft^!][\hbar]$.

\begin{equation} \tag{\textbf{H}}
	\begin{gathered}\label{eq:sdH}
\text{There is an algebra isomorphism $ B( \A) \cong H^\bullet_{T^! \times \Cx}(Y^!) $ compatible with} \\ \text{ the isomorphisms $ \ft \cong H^2(Y^!) $ and $ H^2(Y) \cong \ft^! $ from (\ref{eq:sdT}). }
\end{gathered}
\end{equation}
If we forget the equivariance, then this statement can be simplified to
$ X^{\Cx} \cong \Spec H^\bullet(Y^!)$, which was the form originally noted by Hikita.

\subsection{Enumerative geometry}
A further topic is the relationship between the enumerative geometry of symplectic dual pairs.  Recall in section \ref{se:QC}, we discussed the quantum cohomology of symplectic resolutions and the notion of K\"ahler roots.

Let $ Y \rightarrow X $ be a Hamiltonian symplectic resolution.  The \textbf{equivariant roots} of $ Y $ are the weights of $ T $ on the tangent spaces at the points of $ Y^T $.  These equivariant roots are useful since they appear in the denominators in the character formula for standard modules for quantizations $ A_\theta $ (see \cite[Prop. 5.20]{BLPW}).

The following statement was first formulated in \cite[Section 3.1.8]{okounkov2015enumerative}.

\begin{equation} \tag{\textbf{R}} \label{eq:sdR}
	\begin{gathered}
\text{Under the isomorphism  $ \ft^* \cong H_2(Y^!) $ from (\ref{eq:sdT})} \\
\text{the equivariant roots for $Y$ coincide with the K\"ahler roots for $Y^!$}
\end{gathered}
\end{equation}

Building on (\ref{eq:sdH}) and (\ref{eq:sdR}), in \cite{qh} we formulated a quantum Hikita conjecture, relating to the $D$-module of characters for $ A_\theta $ to 
a specialization of the quantum connection for $ Y^!$.

In another direction, Aganagic-Okounkov \cite{AO} introduced the notion of elliptic stable envelopes and formulated a conjectural statement relating the elliptic stable envelopes of symplectic dual pairs.  It is worth noting that in this statement (unlike all the above statements), the two dual resolutions play a symmetric role.

\subsection{Symplectic duality for hypertoric varieties} \label{se:sdhv}
For hypertoric varieties all of the above statements are true and have been thoroughly investigated.

We return to the setting of section \ref{se:hypertoric} and assume that we are given two dual short exact sequences of tori
$$
1 \rightarrow G \rightarrow (\Cx)^n \rightarrow T \rightarrow 1 \qquad 
1 \rightarrow T^\vee \rightarrow (\Cx)^n \rightarrow G^\vee \rightarrow 1$$
and we will write $ G^! := T^\vee, T^! := G^\vee $.  We also fix $ \chi \in \Hom(G, \Cx) = \Hom(\Cx, T^!) $ and $ \rho \in \Hom(\Cx, T) = \Hom(G^!, \Cx) $. 

We will use this data to define two hypertoric varieties $ Y := T^* \C^n \ssslash_{0,\chi} G $ and $ Y^! = T^* \C^n \ssslash_{0,\rho} G^! $, along with Hamiltonian $ \Cx $ actions on each.  Note that $ \chi $ is used as a GIT parameter to define $ Y $ and in order to define the $ \Cx $ action on $Y^!$. This is part of a general situation in symplectic duality where the data of the resolution $Y$ (for fixed $ X$) matches the data of the choice of $ \Cx \subset T^! $ on the dual side.

As in section \ref{se:hypertoric}, we have hyperplane arrangements in $ \ft^*_\R $ and $ \fg_\R^*$ which are used to study $ Y $ and $ Y^!$.  Moreover, these arrangements are \textbf{polarized}, which simply means that we have the extra data of the covectors $ \rho $ and $ \chi $, respectively.  These two polarized hyperplane arrangements are related by the combinatorial operation of Gale duality, see \cite{BLPWGale}.  

\begin{Example}The Gale dual arrangements for $Y = T^* \PP^2$ and $ Y^! = \widetilde{\C^2/\Z_3} $ are given in Figure \ref{fig:hyp}.
\end{Example}

\begin{Theorem}
	The statements (\ref{eq:sdT}), (\ref{eq:sdS}), (\ref{eq:sdC}), (\ref{eq:sdO}), (\ref{eq:sdH}), (\ref{eq:sdR}) all hold for $ Y, Y^! $.
\end{Theorem}
\begin{proof}
	
	\begin{enumerate}
		\item[(\ref{eq:sdT})] The isomorphism $ \ft_\Z \cong H^2(Y^!,\Z) $ follows from Kirwan surjectivity.  The matching of the attracting cone with the ample cone follows from \cite[Section 8]{BLPWhO}.

\item[(\ref{eq:sdS})] The symplectic leaves of $ Y $ are given combinatorially by the \textbf{coloop-free flats} of the hyperplane arrangement in $ \ft^*_\R $.  This leads to order reversing bijections between the symplectic leaves of $ Y $ and $Y^!$, see \cite[Theorem 10.8]{BLPW}.  These leaves and their transversal slices are themselves hypertoric varieties, associated to \textbf{restriction} and \textbf{localization} of the hyperplane arrangement.  Under Gale duality, restriction and localization are exchanged \cite[Lemma 2.6]{BLPWGale}.  This establishes (\ref{eq:sdS}), see \cite[Example 10.19]{BLPW}.

\item[(\ref{eq:sdC})] As mentioned earlier, the core components of $ Y $ are given by the compact chambers in the hyperplane arrangment.  On the other hand, the components of $ X^+ $ are given by the $ \rho$-bounded maximal cones in the corresponding central hyperplane arrangement.  By \cite[Theorem 5.25]{BLPWGale}, these chambers and cones are exchanged under Gale duality.  This gives the desired bijection $ \Irr X^+ \cong \Irr F_0^! $, establishing (\ref{eq:sdC}).

\item[(\ref{eq:sdO})] The quantization of $ Y $, called the hypertoric enveloping algebra, has been extensively studied in \cite{BLPWhO}, where the desired Koszul duality was established.

\item[(\ref{eq:sdH})] The original (non-equivariant) Hikita statement for hypertoric varieties was established by Hikita in \cite{Hikita}.  The more general equivariant statement doesn't seem to have been carefully written in the literature.

\item[(\ref{eq:sdR})] Finally, the enumerative geometry of hypertoric varieties has been extensively studied.  The statement about roots follows immediately from \cite{McBreenShenfeld}.  In \cite{qh}, we established our quantum Hikita conjecture, including (\ref{eq:sdR}), for symplectic dual hypertoric varieties.  Moreover, Smirnov-Zhou \cite{SZ} studied elliptic stable envelopes for hypertoric varieties establishing the desired dual relationship.
\end{enumerate}
\end{proof}

\section{Geometrization and categorifications of representations} \label{se:SDQV}
In this section, we will see how symplectic duality can be used to explain the relation between two geometric constructions of representations of simple Lie algebras.

Let $ \fg $ be a simple simply-laced Lie algebra (in other words of type $ADE$).  Let $ P $ denote its weight lattice and $ P_+ $ the set of dominant weights.  Let $ \lambda \in P_+$ and write $ \lambda = \lambda_1 + \cdots + \lambda_n $ where all $ \lambda_k $ are fundamental.  We will be interested in the tensor product representation 
$$
V(\ul) := V(\la_1) \otimes \cdots \otimes V(\la_n) 
$$
We can decompose this representation into irreducible representations
$$
V(\ul) = \bigoplus_{\nu \in P_+} \Hom_\fg(V(\nu), V(\ul)) \otimes V(\nu)
$$
Fix another weight $\mu$ and take the $\mu$-weight space on both sides giving
\begin{equation} \label{eq:tensorDecomp}
V(\ul)_\mu = \bigoplus_{\nu \in P_+} \Hom_\fg(V(\nu), V(\ul)) \otimes V(\nu)_\mu 
\end{equation}
Note that the only $ \nu $ that occur here will satisfy $ \nu \le \lambda $ and $ \mu^{dom} \le \nu $, where $ \mu^{dom}$ is the dominant translate of $ \mu$.

Our goal will be to see (\ref{eq:tensorDecomp}) as the hyperbolic BBD decomposition (\ref{eq:fund}) for two dual symplectic resolutions.

\begin{Example} \label{eg:sl2}
	For our running example of (\ref{eq:fund}), we take $ \fg = \ssl_2 $, all $ \lambda _k = \omega $ and $ \mu = n \omega - \alpha $. So we have $ V(\ul) = (\C^2)^{\otimes n} $.  Then (\ref{eq:fund}) becomes
	\begin{gather*}
	(\C^2)^{\otimes n}_{n-2} = \Hom_{\ssl_2}(V(n), (\C^2)^{\otimes n}) \otimes V(n)_{n-2} \\ 
	\oplus \Hom_{\ssl_2}(V(n-2), (\C^2)^{\otimes n}) \otimes V(n-2)_{n-2}
	\end{gather*}
	which numerically gives $ n = 1 \cdot 1 + (n-1) \cdot 1 $.
\end{Example}

\subsection{Quiver varieties}  \label{se:qv2}
Let $ Q = (I,E) $ be an orientation of the Dynkin diagram of $ \fg $.  We define two dimension vectors $ \Bv, \Bw $ by the equations $ \lambda = \sum w_i \omega_i $ and $ \lambda - \mu = \sum v_i \alpha_i $.

As in section \ref{se:quiver}, we can use this data to construct a Nakajima quiver variety which we denote by $ Y =  M(\lambda, \mu) $ (for the projective GIT quotient) and $ X = M_0(\lambda,\mu) $.  We write $ M_0(\lambda, \mu)^{reg} $ for the open subset coming from points with trivial stabilizer for the action of the gauge group.

The following results are essentially due to Nakajima \cite{Nak98}, but are formulated in the language of symplectic resolutions in \cite{MN}.

\begin{Theorem} \label{th:Msymp}
	Assume that $ \mu $ is dominant.
	\begin{enumerate}
		\item
	$ M(\lambda,\mu) \rightarrow M_0(\lambda, \mu) $ is a conical symplectic resolution.  
\item 
The symplectic leaves of $ M_0(\la, \mu) $ are $ M_0(\la, \nu)^{reg} $ for $ \mu \le \nu \le \lambda $.
\item  The Hamiltonian torus $T = (\Cx)^{\sum w_i} $ has finitely many fixed points if every $ \lambda_k $ is minuscule.
\end{enumerate}
\end{Theorem}

If $ \mu $ is not dominant, then the map $ M(\lambda, \mu) \rightarrow M_0(\lambda, \mu) $  will fail to be surjective, as the following example illustrates.
\begin{Example} \label{eq:sl2}
	 Consider a quiver with one vertex and no edges, so that $ \fg = \ssl_2$.
 
	 Because of the framing, the quiver variety is still non-trivial and we have
	 $$ \Phi^{-1}(0) = \{(A,B) \in \Hom(V,W) \oplus \Hom(W,V)  : BA = 0\} $$ 
	Taking the quotient by $ GL_V$, we find 
	$$Y = T^* G(v,W) = \{ (U, C) : U \subset W, \dim U = v, CU =0, C W \subseteq U \} $$ 
	if $ v \le w $ (otherwise it is empty).  Here $ G(v,W) $ denotes the Grassmannian of $ v $-dimensional subspaces of $ W = \C^w $.
	
	On the other hand $$ X = \{C \in \End(W) : C^2 = 0, \rk(C) \le v \} $$ 
	 Note that for any $ (U,C) \in T^* G(v,W)$, we have $ \rk(C) \le \max(v,w-v) $ and thus $ Y \rightarrow X $ is surjective if and only if $ v \le w/2 $.  This corresponds to the condition that $ \mu = w \omega - v \alpha = w -2v $ is dominant.
	
\end{Example}

The decomposition $ \la = \la_1 + \dots + \la_n $ corresponds to a decomposition of $ W $ into one-dimensional vector spaces.  We choose $ \rho : \Cx \rightarrow T $ to have different eigenvalues on each of these one-dimensional pieces.  Then $ M(\la, \mu)^+ $ is a Nakajima tensor product variety \cite{NakTen}.  The following result combines \cite{NakTen} and the work of Varagnolo-Vasserot \cite{VV}. 
\begin{Theorem}
There is an isomorphism $ H(M(\la,\mu)^+) \cong V(\ul)_\mu $ fitting into the following diagram, where the first row is (\ref{eq:fund}) and the second row is (\ref{eq:tensorDecomp}).
\begin{equation}
	\begin{tikzcd}
		H(M(\la,\mu)^+) \arrow[equal]{r} \arrow{d}{\sim} & \bigoplus_{\nu} H(M_0(\la, \nu)^+)  \phantom{h} \otimes H(F_\nu) \arrow[d,"\sim", shift left=20] \arrow[d,"\sim",shift left = -5] \\
V(\ul)_\mu \arrow[equal]{r} & \bigoplus_{\nu} \Hom_\fg(V(\nu), V(\ul)) \otimes V(\nu)_\mu
\end{tikzcd} 
\end{equation}
\end{Theorem}

\begin{Example}
\label{eq:sl2quiver}
Let us take a look at Example \ref{eq:sl2} from the quiver variety viewpoint.  We have $ \Bv = 1, \Bw = n $, and
$$ M(\la, \mu) = T^* \Hom(\C, \C^n) \ssslash_{0,\chi} \Cx \cong T^* \PP^{n-1}$$
Then the hyperbolic BBD decomposition from the second part of Example \ref{eg:fund} matches Example \ref{eg:fund}.
\end{Example}

\subsection{Affine Grassmannian slices} \label{se:AGslices}
We will now consider the construction of representations using the geometric Satake correspondence.  Let $ G^\vee $ be the Langlands dual group to $  G$.

Let $ \cK = \C((t)), \cO = \C[[t]] $.  Let $ \Gr = G^\vee(\cK)/G^\vee(\cO) $ be the affine Grassmannian of $ G^\vee $, an ind-projective variety.  

By the definition of Langlands duality, the weights of $ \fg $ are the coweights of $ G^\vee$.  Each such coweight $ \mu $ is a map $ \Cx \rightarrow T^\vee $ and so defines a point $t^\mu $ in $ \Gr $.

For $ \lambda \in P_+$, we let $ \Gr^\lambda = G^\vee(\cO) t^\lambda$.  Its closure $ \overline{\Gr^\lambda}$ is called a \textbf{spherical Schubert variety}.  We have $ \Gr = \sqcup_{\lambda \in P_+} \Gr^\lambda $.

\begin{Example}
	If $ \fg = \ssl_m$, then $ \lambda = (\lambda^1, \dots, \lambda^m) $, with $ \lambda^1 \ge \cdots \ge \lambda^m = 0 $.  Then $G^\vee = PGL_m $ and 
	$$ t^\lambda = \begin{bmatrix} t^{\lambda^1} &  & 0 \\ & \ddots & \\ 0 & & t^{\lambda^m} \end{bmatrix}
	$$
	Then $ \Gr^\lambda$ can be identified with the space of $ \C[x]$ submodules $ M \subset \C[x]^m $ such that $ \C[x]^m/ M \cong \C[x]/x^{\lambda^1} \oplus \cdots \oplus \C[x]/x^{\lambda^m} $. 
\end{Example} 

We can speak about the ``distance'' between two points in the affine Grassmannian.  We write $ d([g_1], [g_2]) = \lambda \in P_+ $ if $[g_1^{-1} g_2] \in \Gr^\lambda $. 
Given a sequence $ \ul = \lambda_1, \dots, \lambda_n$ of dominant coweights, we consider the polyline variety
$$
\Gr^{\ul} := \Gr^{\la_1} \tilde{\times} \cdots \tilde{\times} \Gr^{\la_n} := \{(L_1, \dots, L_n) : d(L_{k-1}, L_k) = \lambda_k, \text{ for } k = 1, \dots, n \}
$$
where we fix $ L_0 = t^0 $. We have a convolution morphism $ \pi_\ul : \Gr^\ul \rightarrow \Gr $ taking $ (L_1, \dots, L_n) $ to $ L_n $.

\begin{figure}
	\includegraphics[trim=0 440 0 80, clip, width=0.9\textwidth, angle=0]{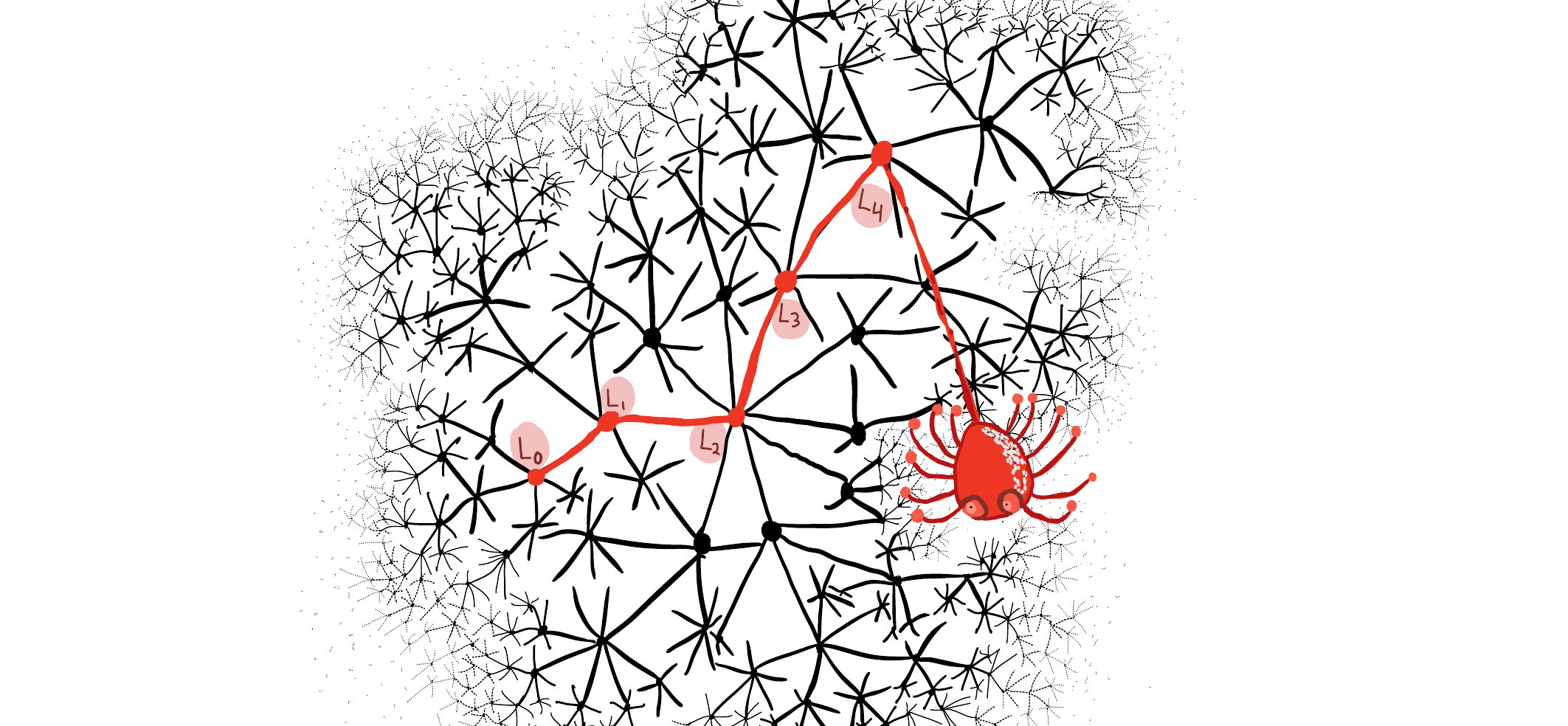} 
	\caption{A polyline in the affine Grassmannian.  For $ G^\vee=PGL_2$, the affine Grassmannian is an infinitely branching tree.}
\end{figure}

Let $ U^\vee \subset G^\vee $ be the maximal unipotent subgroup and let $ S^\mu = U^\vee(\cK) t^\mu $, for $ \mu \in P$.  Alternatively, $ S^\mu $ is the attracting set for a $ \Cx $ action on $\Gr $ defined by a generic dominant $ \Cx \subset T^\vee $.   

The following result is called the geometric Satake correspondence and was proven by Mirkovi\'c-Vilonen \cite{MV}, building on previous work by Lusztig \cite{L} and Ginzburg \cite{G}.
\begin{Theorem} \label{th:Satake}
\begin{enumerate}	
	\item There is an equivalence of category $ P_{G^\vee(\cO)}(\Gr) \cong \Rep \, G $ between the category of perverse sheaves on $ \Gr $ constructible with respect to the $ \Gr^\lambda $ and the category of representations of $G $. 
	\item This equivalence takes $ (\pi_{\ul})_* \IC_{\overline{\Gr^\ul}} $ to $ V(\ul) $.
	\item Under this equivalence, the weight functor $ \Rep \, G \rightarrow \Vect $ defined by $ V \mapsto V_\mu $ corresponds to the hyperbolic stalk functor
	$$ 
	 P_{G^\vee(\cO)}(\Gr) \rightarrow \Vect \quad \mathcal F \mapsto H^\bullet(S^\mu, s_\mu^! \mathcal F)
	 $$
	 where $ s_\mu : S^\mu \rightarrow \Gr $ is the inclusion (this hyperbolic stalk is concentrated in a single degree as in Theorem \ref{th:hypess}).
	 \end{enumerate}
\end{Theorem}
As a corollary of this theorem, we have an isomorphism $ H(\overline{\Gr^\lambda} \cap S^\mu) \cong V(\lambda)_\mu $.

In order to put this in the context of symplectic resolutions, we need to work with affine Grassmannian slices.  For $ \mu \in P_+ $, we define the transversal slices to the spherical Schubert varieties by
$$
\cW_\mu := G^\vee_1[t^{-1}] t^\mu \text{ where } G^\vee_1[t^{-1}] = \ker (G^\vee[t^{-1}] \xrightarrow{t^{-1} \mapsto 0} G^\vee)
$$
Then we set $ \oW^\la_\mu := \overline{\Gr^\la} \cap \cW_\mu $ and $ \cW^{\ul}_\mu := \Gr^\ul \cap \pi_\ul^{-1} (\cW_\mu)$.  Here, as before, $ \lambda = \lambda_1 + \dots + \lambda_n $ and $ \mu \le \lambda $.

\begin{Example}
	We can realize the nilpotent cone of $ \ssl_n $ as an affine Grassmannian slice, following an observation originally due to Lusztig \cite{L}.  Take $ G^\vee = SL_n $ and consider the map 
	$$\cN_{\ssl_n} \rightarrow G^\vee_1[t^{-1}] \quad A \mapsto I + t^{-1} A  $$
	This map defines an isomorphism between $ \cN_{\ssl_n} $ and $ \oW_0^{n\omega_1}$ and can be lifted to an isomorphism $ T^* Fl_n \cong \oW^{\omega_1, \dots, \omega_1}_0 $.
	
	This isomorphism restricts to an isomorphism between nilpotent orbit closures and affine Grassmannian slices. More generally, we can realize any $ SL_n $ affine Grassmannian slice as a type A Slodowy slice, via the Mirkovic-Vybornov isomorphism \cite{MVy}.
\end{Example}

The following results can be found in \cite{KWWY}.
\begin{Theorem} \label{th:Wsymp}
	Assume that all $ \la_k $ are minuscule and $ \mu \in P_+$.
	\begin{enumerate}
\item		There is a Poisson structure on $ \cW_\mu $ coming from the Manin triple $ (\fg\otimes \cK,\fg\otimes \cO, \fg \otimes t^{-1}\C[t^{-1}]) $ and each $ \oW^\la_\mu $ is a Poisson subvariety.
		\item $\pi: \cW^\ul_\mu \rightarrow \oW^\la_\mu $ is a symplectic resolution.
		\item It carries a conical $ \Cx $ action (by ``loop rotation'', acting on the variable $ t $) and a Hamiltonian $ T^\vee $ action with finitely many 
		fixed points.
		\item The symplectic leaves of $ \oW^\la_\mu $ are $ \cW^\nu_\mu $ for $ \mu \le \nu \le \la $.  The transversal slice to $ \cW^\nu_\mu $ is $ \oW^\lambda_\nu $.
		\item We have $ (\oW^\la_\mu)^+ = \overline{\Gr^\lambda} \cap S^\mu$.
	\end{enumerate}
\end{Theorem}

\begin{figure}
	 \includegraphics[trim=0 560 0 100, clip, width=1.1\textwidth, angle=0]{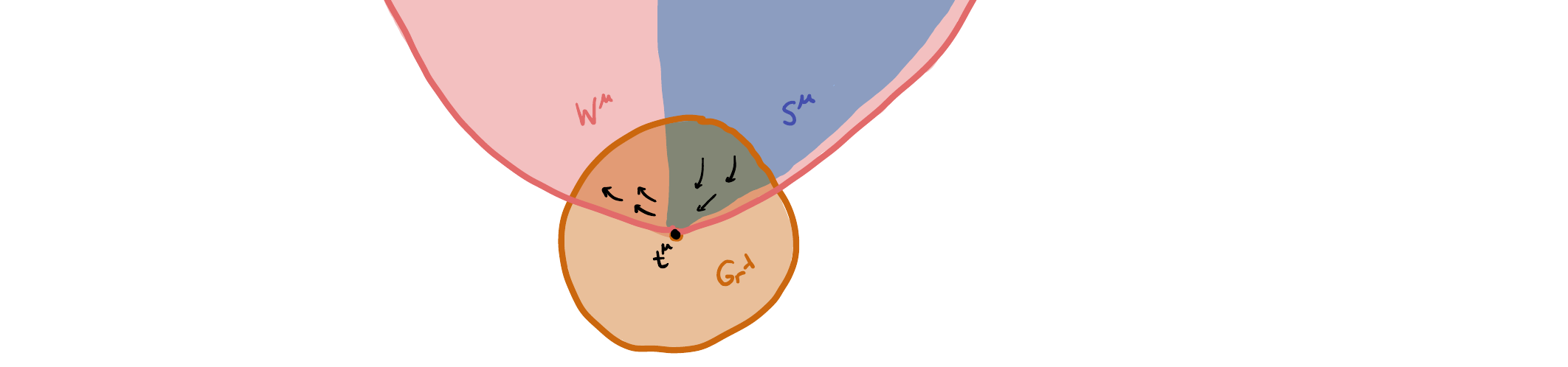} 
\caption{The affine Grassmannian slice $ \oW^\la_\mu $ contains the MV locus $ \overline{\Gr^\lambda} \cap S^\mu $ as its attracting locus.}
\end{figure}

As a consequence of Theorems \ref{th:Satake} and \ref{th:Wsymp}, we have the following result.

\begin{Theorem} \label{th:SatakeSlice}
		There is an isomorphism $  H((\cW^\ul_\mu)^+) \cong V(\ul)_\mu $ fitting into the following diagram, where the first row is (\ref{eq:fund}) and the second row is (\ref{eq:tensorDecomp}).
		\begin{equation}
			\begin{tikzcd}
				H((\cW^\ul_\mu)^+) \arrow[equal]{r} \arrow{d}{\sim} & \bigoplus_{\nu} H((\oW^\nu_\mu)^+) \otimes  H(m^{-1}(t^\nu))  \phantom{h} \arrow[d,"\sim", shift left=10] \arrow[d,"\sim",shift left = -15] \\
				V(\ul)_\mu \arrow[equal]{r} &  \bigoplus_{\nu} \phantom{h} V(\nu)_\mu \otimes \Hom_\fg(V(\nu), V(\ul)) 
			\end{tikzcd} 
		\end{equation}
\end{Theorem}
For the generalization to non-dominant $ \mu$, we will need the generalized affine Grassmannian slices, see section \ref{se:nondominant}.

A Poisson deformation of $ \oW^\la_\mu $ can be constructed using the Beilinson-Drinfeld Grassmannian \cite{KWWY}.  Quantizations of $ \oW^\lambda_\mu $ are constructed using \textbf{truncated shifted Yangians}, defined in \cite{KWWY}.  For each $ \mu \in P_+$, there is an subalgebra $ Y_\mu \subset Y $ of the Yangian, quantizing $ \cW_\mu $.  Then there is a quotient $ Y^\lambda_\mu $ of $ Y_\mu$ which quantizes $ \C[\oW^\la_\mu] $.  This quotient is defined as the image of $ Y_\mu $ in a certain algebra of difference operators (conjecturally, we have an explicit presentation).

\begin{Example}
Let us take a look at Example \ref{eq:sl2} from the affine Grassmannian slices viewpoint. Then $ \oW^\la_\mu = \C^2/\Z_n $ and $ \cW^{\ul}_\mu = \widetilde{\C^2/\Z_n} $ and the hyperbolic BBD decomposition is given in the first part of Example \ref{eg:fund}.  In particular, the $n-1$ $\PP^1$s in the core are responsible for $ \Hom_{\ssl_2}(V(n-2), (\C^2)^{\otimes n}) $.
\end{Example}

\subsection{Symplectic duality between quiver varieties and slices}
We will now check the different aspects of the symplectic duality.  We fix $ \lambda, \mu \in P_+ $ and write $ Y = M(\la, \mu), Y^! = \cW^\ul_\mu $ for some decomposition $ \lambda = \lambda_1 + \dots + \lambda_n $ where all $ \lambda_k $ are minuscule.  As in section \ref{se:sdhv}, this decomposition of $ \lambda $ determines the resolution $ Y^! $ and determines the choice of $ \Cx \subset T $ acting on $ Y $.

\subsubsection{Hamiltonian torus and second cohomology}
First, recall that we have $ (\Cx)^{\sum w_i} $ acting on $ M(\la,\mu)$.  This action is not necessarily faithful (in particular the diagonal $ \Cx $ acts trivially) and we write $ T $ for its image in the automorphism group of $ M(\la,\mu) $.   So $ \ft $ is a quotient of $ \C^n$.  (Here as before $ n $ is the size of the list $ \ul $ and also equals $ \sum w_i $.)

On the other hand, by definition, we have an embedding $ \cW^\ul_\mu \subset \Gr^n $ and thus a map
$$
\C^n = H^2(\Gr)^{\oplus n}= H^2(\Gr^n) \rightarrow H^2(\cW^\ul_\mu) $$
Combining this should establish $ \ft \cong H^2(\cW^\ul_\mu) $.

The reverse direction is simpler.  We have a Kirwan map $ \C^I \rightarrow H^2(M(\la,\mu))$ which is an isomorphism (as long as $ v_i \ne 0 $ for all $ i$). Also, we have an action of $ (\Cx)^I = T^! $ on $ \cW^\ul_\mu $.  This gives us the desired isomorphism $ \ft^! \cong H^2(M(\la, \mu)) $, establishing (\ref{eq:sdT}) in both directions.

\subsection{Matching of strata and cycles}
From Theorems \ref{th:Msymp} and \ref{th:Wsymp}, the symplectic leaves of $ M_0(\la, \mu) $ and $ \oW^\la_\mu $ are both labelled by dominant $ \nu $ such that $ \mu \le \nu \le \la $.  So we get the desired bijection between leaves, which is seen to be order-reversing.  Moreover, the symplectic leaves of $ M_0(\la,\mu) $ are $ M_0(\la, \nu)^{reg} $ and the transversal slices in $ \oW^\la_\mu $ are $ \oW^\la_\nu $, and  we have the symplectic duality between $ M(\la, \nu) $  and $ \cW^\ul_\nu $.  This establishes (\ref{eq:sdS}).

The irreducible components of $ (\oW^\la_\mu)^+ = \overline{\Gr^\la} \cap S^\mu  $ are known as Mirkovic-Vilonen cycles \cite{MV}.  In \cite{mvcycle}, we constructed a bijection between MV cycles and a collection of polytopes, known as MV polytopes.  On the other hand, in \cite{BK}, we constructed a bijection between these same polytopes and the components of the cores of Nakajima quiver varieties.  Combining these two bijections leads to a bijection
$$
\Irr F_0 \cong \Irr (\oW^\la_\mu)^+ 
$$
The reverse direction $ \Irr M(\la, \mu)^+ \cong F^!_0 $ is not as well-developed.

\subsection{Categorical} \label{se:KDslices}
As the quiver variety $ M(\la, \mu) = T^* N \ssslash_{0,\chi} G $ is constructed by Hamiltonian reduction, any quantization $A_\theta$ can be constructed by quantum Hamiltonian reduction.  Thus, the category of $A_\theta$-modules can be described using a category of $ G$-equivariant $ D_N$-modules (see \cite[Prop 2.6]{W1409}).  This allows one to study these modules using the methods of $ D$-modules.  In particular, Webster \cite[Theorem A']{W1409} proved that (for generic integral $\theta$) the graded lift $\tilde \cO$ of category $ \cO $ for $ A_\theta$ is equivalent to the graded category
	of linear projective complexes over an explicit diagrammatic algebra, called a Khovanov-Lauda-Rouquier-Webster algebra $ T^\la_\mu$. Moreover, he proved that this category is Koszul.
	
 On the other hand, the quantization of an affine Grassmannian slice $ \oW^\la_\mu $ is given by the truncated shifted Yangian $Y^\la_\mu$. In \cite{KTWWY}, we proved by an explicit algebraic computation that $ \Mof{T^\la_\mu}$ is equivalent to the category $\cO $ for $ Y^\la_\mu$. Combining all these results yields the desired Koszul duality (\ref{eq:sdO}), see \cite[Corollary 1.3]{KTWWY}.

\subsection{Hikita}
In \cite{moncrystals}, we constructed an isomorphism 
$$ H^\bullet(M(\la, \mu)) \cong \C[(\cW^\lambda_\mu)^{\Cx}]$$
which verifies the original (non-equivariant) Hikita statement.  To prove this, we identified both sides with the weight space of a representation of a current Lie algebra $ \fg^\vee[[t]]$.

On the other hand, let us consider the full equivariant version (\ref{eq:sdH})
$$ H^\bullet_{T \times \Cx}(M(\la, \mu)) \cong B(Y^\lambda_\mu)$$
  In \cite{moncrystals}, we proved that, up to nilpotents, this is equivalent to the highest weights for $ Y^\lambda_\mu $ being given by the combinatorics of the product monomial crystals.  In \cite{KTWWY}, we proved this statement about the highest weights, and thus (\ref{eq:sdH}) is almost known.
  
  The reverse direction, relating $ H^\bullet(\cW^\ul_\mu) $ with the $B$-algebra of the quiver variety, has not been explicitly studied in the literature.

\section{Coulomb branches of 3d gauge theories} \label{se:CB}
In general, there is no obvious geometric relationship between two symplectic dual varieties, nor any obvious way to construct one from the other.  However, when one symplectic resolution is constructed by Hamiltonian reduction, then Braverman-Finkelberg-Nakajima \cite{BFN1} gave a method for constructing the symplectic dual, motivated by quantum field theory.

\subsection{The Coulomb branch of 3d gauge theory}
Fix a reductive group $ G $ (the ``gauge group'') and a representation $ N$ (the ``matter'').  This data defines a 3d $N=4$ supersymmetric gauge theory.  This gauge theory has various moduli of vacua, one of which is the Higgs branch, which is the Hamiltonian reduction of $ T^*N $, as explained in section \ref{se:Higgs}.

Another modulus of vacua is called the Coulomb branch, denoted $ M_C(G,N)$.  As noted in section \ref{se:3dmirror}, the Higgs and Coulomb branches are expected to be symplectic dual.  The Coulomb branch is expected to have the structure of a hyperk\"ahler manifold, so from the algebraic geometry perspective it should be a complex symplectic variety.  

\subsection{The BFN Coulomb branch}
A mathematical definition of $M_C(G,N)$ was given by Braverman-Finkelberg-Nakajima \cite{BFN1}.  Their approach is to first define an algebra $ \bar A(G,N) $ and then define the Coulomb branch as $\Spec \bar A(G,N) $.  This work was partially inspired by the work of Cremonesi-Hanany-Zaffaroni \cite{CHZ}, who gave a physics computation of the Hilbert series of $\bar A(G,N) $. For further explanation of the physics motivation, see \cite{NakBFN}.

Let $ D = \Spec \cO $, be the formal disk, and $ \Dx = \Spec \cK $, the formal punctured disk.  We consider the non-separated curve $ \BB := D \cup_{\Dx} D $, sometimes called the \textbf{raviolo} or \textbf{bubble}.  Then we define 
\begin{equation} \label{eq:CoulombQuickDef}
\bar A(G,N) := H_\bullet(\Maps(\BB, [N/G])) 
\end{equation}
as the homology of the moduli space of maps from the raviolo to the stack quotient $ [N/G]$.  This homology group will carry a convolution algebra structure related to stacking of ravioli. A map to $ [N/G] $ is equivalent to a principal $ G $ bundle $ \cP$ on $ \BB $ along with a section of the associated $ N $ bundle $ N_\cP := \cP \times^G N $. 

	\begin{figure}
	\includegraphics[trim=210 310 230 55, clip,width=0.3\textwidth]{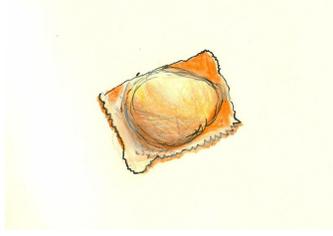}
	\caption{The raviolo curve is formed by gluing two copies of the formal disk along the punctured disk.}
\end{figure}

We can unpack this definition by trivializing our bundle over one disk.  First we consider the moduli space
\begin{align*}
T_{G,N} :&= \{(\cP, \varphi, s) : \cP \text{ is a principal $ G $-bundle on $ D $},\\ 
&\phantom{12345} \varphi : \cP|_{\Dx} \rightarrow \cP^0|_{\Dx} \text{ is a trivialization and $ s \in \Gamma(D, N_\cP)$} \} \\
&= \{([g], v) \in \Gr \times N\otimes \cK : v \in gN \otimes \cO \}
\end{align*}
So $ T_{G,N} $ is a vector bundle over the affine Grassmannian with fibre over $ [g] $ given by $ g N \otimes \cO $.  

We have a natural map $ T_{G,N} \rightarrow N\otimes \cK $, which is analogous to the Springer resolution $ T^* G/B \rightarrow \mathcal N $.  Building on this analogy, it is natural to form the fibre product $ T_{G,N} \times_{N \otimes \cK} T_{G,N} $, in order to define a convolution algebra.  To reduce the infinite dimensionality of this situation, Braverman-Finkelberg-Nakajima take a slightly different approach and consider
$$
R_{G,N} :=  \{([g], v) \in T_{G,N} : v \in N \otimes \cO \}
$$
The map $ \pi : R_{G,N} \rightarrow \Gr$ is no longer a vector bundle, but still restricts to vector bundle over each $ \Gr^\lambda $.

\begin{Example}
	Take $ G = GL_n$ and $ N = \C^n$.   Then we can identify the above players as follows.
	\begin{gather*}
		\Maps(\BB, [N/G]) = \{(\cV, s) : \text{ $\cV $ is a rank $ n $ vector bundle on $ \BB $}, s \in \Gamma(\BB, \cV) \} \\
		\Gr_G = \{ L \subset \cK^n : L \text{ is a $ \cO$ lattice } \} \\
		T_{G,N} = \{ (L, v) : L \in \Gr, v \in L \} \\
		R_{G,N} = \{ (L, v) : L \in \Gr, v \in L \cap \cO^n \}
	\end{gather*}
\end{Example}
		
From the modular viewpoint, $ R_{G,N} $ is the moduli space of triples $ (\cP, \varphi, s) $ as in the definition of $ T_{G,N}$, except that $ s $ is required to be a common section of $N_\cP $ and the trivial bundle $N \otimes \mathcal O_D $.  Hence, the space $R_{G,N}$ has an action of $G(\cO) $ and the mapping space $\Maps(\BB, [N/G]) $ can be identified as the quotient $R_{G,N} / G(\cO)$. Thus, as a more precise version of (\ref{eq:CoulombQuickDef}), we define
$$
\bar A(G,N) := H^\bullet_{G(\cO)}(R_{G,N}, \omega_R)
$$
where $ \omega_R $ is a renormalized dualizing sheaf.  The effect of this dualizing sheaf is that we work with cycles that are finite-dimensional along $ \Gr $ and infinite-dimensional along the fibres.  For example, we can consider $ r_\lambda = [R^\lambda] $ where $ R^\lambda$ is the closure of the preimage of $ \Gr^\lambda $ ($r_\lambda $ is called a \textbf{monopole operator} in the physics literature \cite{CHZ}). 

Braverman-Finkelberg-Nakajima \cite{BFN1} prove the following result by adapting previous work of Bezrukavnikov-Finkelberg-Mirkovic \cite{BFM} on the equivariant homology of the affine Grassmannian.
\begin{Theorem}
	Using a convolution diagram, $ \bar A(G,N) $ carries the structure of a commutative algebra. It is a finitely-generated, integrally closed, integral domain.
\end{Theorem}

We define $ M_C(G,N) = \Spec \bar A(G,N) $.  From the theorem, it is a normal affine variety.  

\begin{Example} \label{eg:AC}
	Fix an integer $ n \ge 1$.	Take $ G = \Cx $ and $ N = \C(n) $, a 1-dimensional vector space on which $ \Cx $ acts by weight $n$.
	
	Then $ \Gr = \cK^\times / \cO^\times = \Z $ and $ R_{G,N} = \bigcup_{k \in \Z}  \{t^k\} \times (\cO \cap t^{nk} \cO)  $.  
	
	Thus, as a vector space, 
	$$ \bar A(G,N) = H_\bullet^{G(\cO)}(R_{G,N}) = \bigoplus_{k \in \Z} \C[w]r_k$$
	where $r_k = [\{t^k\} \times ( \cO \cap t^{nk}\cO)]$ is a monopole operator, and where $ \C[w] = H^\bullet_{\Cx}(pt) $. 
	
	So it remains to understand the multiplication.  Examining the definition of the convolution diagram, we see that $ r_k r_l $ is determined by the
 	intersection 
$$	(\cO \cap t^{nk}\cO) \cap (t^{nk} \cO \cap t^{n(k+l)} \cO)$$
inside $ t^{nk} \cO $.  This intersection 
is transverse when $ k, l$ have the same sign.  If $ k < 0, l > 0 $, we get an excess intersection of $ t^{nk} \cO / \cO + t^{n(k+l)} \cO $ and if $ k > 0, l < 0 $, the intersection is smaller than the expected space $ \cO \cap  t^{n(k+l)} \cO $.  In either case, this leads to an equivariant factor $ w^{np} $ where $ p = \min(|k|, |l|) $.  So the multiplication is given by
$$ r_k r_l = \begin{cases} r_{k+l} \text{ if $ k l \ge 0 $ } \\
	w^{np} r_{k+l} \text{ if $ kl < 0 $ }
	\end{cases}
$$
Thus we conclude that $ \bar A(G,N) = \C[r_+, r_-, w]/ (r_+ r_- - w^n) $ and thus $ M_C(G,N) = \C^2 / \Z_n$.
\end{Example}

\subsection{Integrable system}
The Coulomb branch possesses one special feature not discussed earlier for a general symplectic resolution.  The map $ H^\bullet_G(pt) \rightarrow \bar A(G,N) $ given by acting on $ r_0 $ gives us a complete integrable system $ M_C(G,N) \rightarrow \fh / W $.  The fibres of this integrable system are easy to analyse.  In particular the general fibre is isomorphic to $ H^\vee $.  This is in line with the physics expectation that $ M_C(G,N)$ is approximated by $ T^* H^\vee/ W$.  (Here $ \fh $ denotes the Cartan subalgebra of $ \fg $ and $ H $ the maximal torus of $ G $.)

\subsection{Quantization and deformation} \label{se:CoulombQuant}
It is very easy to define the quantization of the Coulomb branch.  There is a loop rotation action of $ \Cx $ on $ R_{G,N} $ scaling the variable $ t $.  (In terms of principal $G$-bundles, this comes from $ \Cx $ acting on the disk $ D $.)  Then we define
$$
A(G,N) := H^\bullet_{G(\cO) \rtimes \Cx}(R_{G,N}, \omega_R)
$$
This is a $ \C[\hbar] $ algebra, where $ \C[\hbar] = H^\bullet_\Cx(pt)$.  The  algebra $A(G,N)$ is non-commutative and $ A(G,N) / \hbar A(G,N)  = \bar A(G,N)$.  This endows $ \bar A(G,N)$ with the structure of a Poisson algebra as in section \ref{se:quantize} and thus $ M_C(G,N) $ is a Poisson variety.

The integrable system on $ M_C(G,N) $ mentioned above quantizes to a subalgebra $ H^\bullet_{G\times \Cx}(pt) \subset A(G,N) $ which we call the \textbf{Gelfand-Tsetlin subalgebra}\footnote{It was called the Cartan subalgebra in \cite{BFN1}, but we prefer the name Gelfand-Tsetlin because of Example \ref{eg:usualGT}.}.

\begin{Example}
	We continue Example \ref{eg:AC}.  Because of the non-trivial $ \Cx $ action on $ \cK $, the excess intersection $ t^{-n} \cO / \cO $ carries a non-trivial $ \Cx $ action and we see that in $ A(G, N)$, we have
	$$ r_+ r_- = w(w + \hbar) \cdots (w + (n-1) \hbar)  \quad r_- r_+ = (w - \hbar)  \cdots (w -n \hbar)$$
	In this case, $ \C[w,\hbar] $ is the Gelfand-Tsetlin subalgebra.
\end{Example}

In order to consider the commutative deformation of the Coulomb branch, we assume that we are given an extension $ 1 \rightarrow G \rightarrow \tG \rightarrow T \rightarrow 1 $ and an action of $ \tG $ on $ N $.  Then $ \tG(\cO) $ acts on $ R_{G,N} $ and we can form
$$
\bar \A(G,N) = H^\bullet_{\tG(\cO)}(R_{G,N}, \omega_R) \quad \A(G,N) := H^\bullet_{\tG(\cO) \ltimes \Cx}(R_{G,N}, \omega_R)
$$
The map $ \Spec \bar \A(G,N) \rightarrow \Spec H^\bullet_T(pt) $ provides a deformation of $ M_C(G,N) $ over $ \ft $ and $ \A(G,N) $ is a family of quantizations over the same space.

\begin{Example} \label{eg:usualGT}
	Consider $ G = GL_1 \times \cdots \times GL_{n-1} $ and $ N = \oplus_{k=1}^{n-1} \Hom(\C^k, \C^{k+1})$ as in Example \ref{eg:Anquiver}.  Then we have $ \A(G,N) \cong \tilde U_\hbar \mathfrak{sl}_n $ (this is a special case of \cite[Theorem B.18]{BFN2}).  The Gelfand-Tsetlin subalgebra $ H^\bullet_{G\times \Cx}(pt) $ discussed above coincides with the classical Gelfand-Tsetlin subalgebra of $  \tilde U_\hbar \mathfrak{sl}_n$ (the subalgebra generated by the centres of all $ U_\hbar \mathfrak{sl}_k $ for $ k \le n $).
\end{Example}

\subsection{Conical and Hamiltonian actions} \label{se:actionCoulomb}
In order to define torus actions on the Coulomb branch, we need gradings on $ \bar A(G,N)$.  The first natural grading is the cohomological grading, which should give the conical $ \Cx $-action.  However, for this action to be conical, we need to shift the grading.  For each $ \lambda \in P_+$, there is a certain half-integer $ \Delta(\lambda) $, defined in (2.10) in \cite{BFN1}, and then we shift the grading by placing $ r_\lambda $ in degree $ 2\Delta(\lambda)$.  The pair $ G,N$ is called \textbf{good or ugly} if $ 2\Delta(\lambda) \ge 1 $ for all dominant coweights $ \lambda$.  In this case, this shifted grading is positive and so we get a conical $ \Cx $ action on $ M_C(G,N)$.  (In the ``ugly'' case, $ 2\Delta(\lambda) = 1 $ for some $ \lambda $, which means that this grading violates the assumption mentioned in Remark \ref{re:nodeg1}.)

For the second grading, we note that $ \pi := \pi_0(\Gr_G) = \pi_1(G) $, and for simplicity we assume this is a free abelian group.  Thus we get a decomposition of $ R_{G,N} $, labelled by $ \pi $, and thus a $ \pi $-grading on $ \bar A(G,N)$.  

The $ \pi $-grading on $ \bar A(G,N)$ gives us a $ T^! $ action on $ \bar A(G,N)$, where we define $T^! $ to be the torus with weight lattice $ \pi $ (it is the Pontryagin dual of $ \pi_1(G)$).  This is a Hamiltonian action where the moment map comes from $ \ft^! = H^2_G(pt) \rightarrow \bar A(G,N) $.  

\subsection{Partial resolution} \label{se:CoulombRes}
As defined, the Coulomb branch is an affine scheme, usually singular, so it is natural to think about how to construct its resolution.  Assume that we have the extension $ 1 \rightarrow G \rightarrow \tG \rightarrow T \rightarrow 1 $ as before and consider the Coulomb branch $ M_C(\tG, N) $. By the construction in the previous section, we get a Hamiltonian action of $ T^\vee $ on $ M_C(\tG, N)$ (since the weight lattice of $ T^\vee $ equals the coweight lattice of $ T $ which is a quotient of $ \pi_1(\tG) $).  

From \cite[Prop 3.18]{BFN1}, we know that $ M_C(\tG, N) \ssslash_{0,0} T^\vee = M_C(G,N) $ and thus it makes sense to consider $ M_C(\tG, N) \ssslash_{\chi, 0} T^\vee $ for some character $ \chi : T^\vee \rightarrow \Cx $.  This space will be a candidate for the resolution $ Y $ of $ X = M_C(\tG,N) $.

For the deformation $\tY$ of the resolution, we consider $ M_C(\tG, N) \sslash_\chi T^\vee $.  In other words, we don't take the 0 level of the moment map, similar to Example \ref{eg:deformHiggs}.

As this construction does not always give a symplectic resolution, it is important to know if the Coulomb branch always has symplectic singularities, as conjectured in \cite[\S 3(iv)]{BFN1}.   The strongest result in this direction is due to Weekes \cite{W}.

\begin{Theorem}
	If $ G, N $ is a quiver gauge theory corresponding to a quiver without loops or multiple edges, then $ M_C(G,N)$ has symplectic singularities.
\end{Theorem}

\subsection{Hypertoric varieties as Coulomb branches}
All hypertoric varieties appear naturally as Coulomb branches.  As in section \ref{se:hypertoric}, fix
$$
1 \rightarrow G \rightarrow (\Cx)^n \rightarrow T \rightarrow 1
$$
leading to an action of $ G $ on $ N = \C^n $.  

From the symplectic duality between Higgs and Coulomb branches and the symplectic duality between Gale dual hypertoric varieties (section \ref{se:sdhv}) it is natural to expect the following result, proven in \cite[Section 4(vii)]{BFN1}.
\begin{Theorem}
$ M_C(G,N) $ is the (singular) hypertoric variety defined by the dual sequence $ 1 \rightarrow T^\vee \rightarrow (\Cx)^n \rightarrow G^\vee \rightarrow 1$.
\end{Theorem}
There are two ways to prove this theorem, both relatively simple.  The first approach is to directly compute this algebra following the method of Example \ref{eg:AC}.  The second approach is to first study the case of $ G = (\Cx)^n $ by this direct method and then conclude the general case by applying a result about Hamiltonian reduction of Coulomb branches \cite[Prop 3.18]{BFN1}.

\subsection{Symplectic duality for Higgs and Coulomb branches}
Here are some indications of the symplectic duality between Higgs and Coulomb branches for an arbitrary $ G, N$.  Let $ Y = T^* N \ssslash_{0,\chi} G $ be the Higgs branch and $ Y^! = M_C(G,N) $ be the Coulomb branch.

First of all, suppose we have the extension $ 1 \rightarrow G \rightarrow \tG \rightarrow T \rightarrow 1 $ as before.  Then $ T $ acts on $ Y  $, while on the other hand the Coulomb branch $ Y^! $ admits a deformation $\tY^! $ over $ \ft$.  

On the other hand, we have the Kirwan map $ \Hom(G, \Cx) \rightarrow H^2(Y, \Z) $ and we have an action of $ T^! $ on $ Y^! $ where $ T^! $ is the torus with weight lattice $ \pi_1(G) $.  As $ \pi_1(G)$ is dual to $ \Hom(G,\Cx)$ (when both are torsion-free), this suggests that $ H^2(Y, \Z) $ is the coweight lattice of the torus acting on $ Y^!$.

Together these paragraphs suggest that (\ref{eq:sdT}) holds.

At the moment, nothing general is known about strata and transversal slices for Coulomb branches.  Some interesting speculation in this direction with a view towards (\ref{eq:sdS}) was described by Nakajima \cite{NakQuestions}.

Webster \cite{WKoszul} proved a general Koszul duality result for category $ \cO $ for Higgs and Coulomb branches (the Koszul duality described in section \ref{se:KDslices} is a special case of this result).  This essentially establishes (\ref{eq:sdO}).

Surprisingly nothing is known about the Hikita conjecture (\ref{eq:sdH}) in general, even though the statement (in the reverse direction) fits well with the definition of the Coulomb branch.

Finally, let us mention that we showed in \cite{HKM} that some aspect of the enumerative geometry passes from the Higgs to the Coulomb branch.  We proved that the quantized Coulomb branch $ A(G,N) $ acts on the cohomology of the spaces of based quasi-maps into $ Y $.

A big obstruction to making further general progress is the following open question.
\begin{Question}
	For which $ G, N $ are the Higgs and the (resolved) Coulomb branches both conical symplectic resolutions?
\end{Question}
We are not aware of any examples beyond the toric and quiver examples.

\section{Affine Grassmannian slices as Coulomb branches}
 We will now see that affine Grassmannian slices arise as Coulomb branches associated to quiver gauge theories and that this perspective allows us to generalize these slices.
 
\subsection{Generalized affine Grassmannian slices} \label{se:nondominant}
 Fix a finite directed graph $Q = (I,E)$ without loops (multiple edges are allowed).  Also, we fix two dimension vectors $ \Bv, \Bw \in \N^I $ and obtain a quiver gauge theory $ G, N $ with flavour torus $ T = (\Cx)^{\sum w_i}$.   
 
 Associated to $ Q $ we have a symmetric Kac-Moody Lie algebra $ \fg_Q $.  We use the dimension vectors $ \Bv, \Bw $ to define a dominant weight $ \lambda $ and another weight $ \mu $ for $ \fg_Q $.  We will write $M_C(\lambda,\mu) := M_C(G,N) $ for the Coulomb branch associated to this quiver gauge theory.

 Assume, until the end of this subsection, that $ \fg_Q $ is of finite-type and let $ G_Q $ denote the associated semisimple group. If $ \mu $ is dominant, recall that in section \ref{se:AGslices}, we defined $ \oW^\lambda_\mu $, an affine Grassmannian slice inside the affine Grassmannian of $ G_Q^\vee $.  Its main feature is that $ (\oW^\lambda_\mu)^+ = \overline{\Gr^\lambda} \cap U^\vee((t))t^\mu $ (here $ U^\vee $ is the maximal unipotent subgroup of $ G_Q^\vee$).

If $\mu$ is not dominant, then Bullimore-Dimofte-Gaiotto \cite{BDG} and Braverman-Finkelberg-Nakajima \cite{BFN2} introduced a generalized affine Grassmannian slice defined by
$$
\oW^\lambda_\mu = U^\vee_1[t^{-1}] T^\vee_1[t^{-1}] t^\mu U^\vee_{-,1}[t^{-1}] \cap \overline{G_Q^\vee(\cO) t^\lambda G_Q^\vee(\cO)}$$
where $ U^\vee, U^\vee_- $ are opposite maximal unipotent subgroups of $ G_Q^\vee $  and $ T^\vee $ is the maximal torus of $ G_Q^\vee$.  Note that this intersection takes place in $G^\vee_Q(\cK) $, rather than $ \Gr = G_Q^\vee(\cK)/ G_Q^\vee(\cO)$.

If $ \mu $ is not dominant, then $\oW^\la_\mu$ does not carry a conical $ \Cx $ action, but it does always have a Hamiltonian action of $ T^\vee$.  
\begin{Theorem} \label{le:genSlice}
\begin{enumerate}
	\item If $ \mu $ is dominant, then $ \oW^\lambda_\mu $ defined here agrees with the affine Grassmannian slice defined earlier.
	\item $(\oW^\lambda_\mu)^{T^\vee} $ is non-empty if and only if $ \mu $ is a weight of $V(\lambda) $.  Moreover, if it is non-empty, then $  (\oW^\lambda_\mu)^+ \cong \overline{\Gr^\lambda} \cap S^\mu $  \cite{Krylov}. 
	\item The symplectic leaves  of $ \oW^\lambda_\mu $ are given by $ \cW^\nu_\mu $ for dominant $ \nu $ such that $ \mu \le \nu \le \lambda $ \cite{MW}.
\end{enumerate}
\end{Theorem}

In \cite[Theorem 3.10]{BFN2}, Braverman-Finkelberg-Nakajima proved the following result. 
\begin{Theorem} \label{th:CoulombGr}
	The Coulomb branch $ M_C(\lambda,\mu)$ is isomorphic to the generalized affine Grassmannian slice $ \oW^\lambda_\mu $.
\end{Theorem}

\begin{Remark}
	The group $ G_Q $ is necessarily simply-laced by construction.  However, the notion of Coulomb branch was generalized by Nakajima-Weekes \cite{NW} to add data corresponding to ``symmetrizers'' (in the sense of symmetrizable Cartan matrices).  With this definition, they generalized  Theorem \ref{th:CoulombGr} so that $ G_Q $ could be any semisimple group.
\end{Remark}
\subsection{Geometric Satake conjecture}
Theorem \ref{th:CoulombGr} allows us to define ``affine Grassmannian slices'' associated to any symmetric Kac-Moody Lie algebra $ \fg_Q $, even though there is no affine Grassmannian and geometric Satake correspondences in this setting.

By section \ref{se:actionCoulomb}, $M_C(\lambda, \mu) $ carries an action of $ T^\vee := (\Cx)^I $ since $ \pi_1(G) = \Z^I $.  Theorem \ref{th:SatakeSlice} and Lemma \ref{le:genSlice} motivate us to expect that the attracting set of $ M_C(\lambda, \mu) $ encodes the weight space of the irreducible $ \fg_Q $ representation $ V(\lambda) $.  The following is a reformulation of the conjecture from \cite[Conj 3.25]{BFN2}; see also Finkelberg \cite{Fin} for more on this conjecture.

\begin{Conjecture} \label{co:Sat}
	 $M_C(\lambda,\mu)^{T^\vee} $ is non-empty if and only if $ \mu $ is a weight of $V(\lambda) $.  Moreover, if it is non-empty, then  $  H(M_C(\lambda,\mu)^+) \cong V(\lambda)_\mu$.
	\end{Conjecture}

When $ Q $ is of affine type $A$, then Nakajima-Takayama \cite{NakTak} proved that $ M_C(\lambda, \mu) $ is a \textbf{bow variety}, a certain modification of a quiver variety introduce by Cherkis \cite{Cherkis}.  Nakajima \cite{NakSatake} used this bow variety description in order to prove the above conjecture in this case.

\subsection{Categorical $ \fg_Q $ actions coming from Coulomb branches}
Let $ A(\lambda, \mu)_\theta $ be a quantization of $ M_C(\lambda, \mu)$, as defined in section \ref{se:CoulombQuant} (here $ \theta \in \ft = \C^{\sum w_i}$).  When $ \fg_Q $ is finite-type, then $ A(\lambda, \mu)_\theta $ is isomorphic to a truncated shifted Yangian $Y^\lambda_\mu$ by \cite[Theorem B.18]{BFN2}.

From Conjecture \ref{co:Sat} and (\ref{eq:ccmap}), it is natural to expect that category $ \cO$ for $ A(\lambda, \mu)_\theta $, with $ \mu $ varying, can be used to categorify $ V(\lambda)$, or the associated tensor product representation $V(\ul) $.  (There is also a version of Conjecture \ref{co:Sat} where we use the partial resolution of $ M_C(\lambda, \mu) $ defined in section \ref{se:CoulombRes}.)

In \cite{restaction}, we defined restriction and induction functors between category $ \cO $ for Coulomb branches, generalizing the Bezrukavnikov-Etingof \cite{BE} restriction and induction functors between Cherednik algebras.  Using the relation with KLRW algebras, we proved the following result.
\begin{Theorem}
	Assume that $ \theta $ is integral.  These restriction/induction functors define a categorical $ \fg_Q $ action on $ \bigoplus_\mu \Oof{A(\lambda,\mu)_\theta}$.  This categorifies $ V(\lambda) $ when $ \theta = 0 $ and categorifies $ V(\ul) $ for generic $ \theta $.
\end{Theorem}

\end{document}